%% file: ConollyLike.tex
\documentclass[12pt]{amsart}
\usepackage{hyperref
}
\usepackage{graphicx}

\usepackage{color}
\usepackage{amsmath}
\usepackage{amssymb}
\usepackage{amscd}
\usepackage{latexsym}
\usepackage{verbatim}
\usepackage{algorithmic}
\usepackage{algorithm}
\usepackage{ulem}
\usepackage[left=1in,right=1in,top=1in,bottom=1in]{geometry}

\def\I#1{{[\![#1]\!]}}            
\def\ord{order}                 

\newcommand{\SEQ}[4]{{\langle #1 ; #2 : #3 ; #4 \rangle}}
\newcommand{\cln}[2]{ \left\lceil \frac{#1}{#2} \right\rceil}
\newcommand{\flr}[2]{ \left\lfloor \frac{#1}{#2} \right\rfloor}
\newcommand{\rem}[1]{#1^{(r)}}
\newcommand{\quo}[1]{#1^{(q)}}
\newcommand{\cell}[2]{\begin{tabular}{|l|l|} \hline #1 & #2\\
 \hline \end{tabular}}
\newcommand{\set}[1]{\left\{ #1 \right\}}

\newcommand{\lref}[1]{Lemma \ref{lem:#1}}
\newcommand{\cref}[1]{Corollary \ref{cor:#1}}

\newcommand{\eref}[1]{Equation (\ref{eq:#1})}
\newcommand{\sref}[1]{Section \ref{sec:#1}}
\newcommand{\tref}[1]{Theorem \ref{thm:#1}}
\newcommand{\tabref}[1]{Table \ref{tab:#1}}
\newcommand{\fref}[1]{Figure \ref{fig:#1}}

\numberwithin{equation}{section}
\newtheorem{theorem}{Theorem}[section]
\newtheorem{corollary}{Corollary}[section]
\newtheorem{lemma}{Lemma}[section]

\newtheorem{conjecture}{Conjecture}[section]

\title[Nested Recurrence Relations With Conolly-Like Solutions]
{Nested Recurrence Relations With Conolly-Like Solutions}

\author[Erickson]{Alejandro Erickson}
\address{Dept. of Computer Science, University of Victoria, CANADA}

\author[Isgur]{Abraham Isgur}
\address{Dept. of Mathematics, University of Toronto, CANADA}

\author[Jackson]{Bradley W. Jackson}
\address{Dept. of Computer Science, San Jose State University, USA}

\author[Ruskey]{Frank Ruskey*}
\address{Dept. of Computer Science, University of Victoria, CANADA}
\thanks{*Research supported in part by NSERC}
\urladdr{http://www.cs.uvic.ca/~ruskey}

\author[Tanny]{Stephen M. Tanny}
\address{Dept. of Mathematics, University of Toronto, CANADA}

\begin{document}
\normalem

\begin{abstract}
  A nondecreasing sequence of positive integers is
  \emph{$(\alpha,\beta)$-Conolly}, or \emph{Conolly-like} for short,
  if for every positive integer $m$ the number of times that $m$
  occurs in the sequence is $\alpha + \beta r_m$, where $r_m$ is $1$
  plus the 2-adic valuation of $m$.  A recurrence relation is
  $(\alpha, \beta)$-Conolly if it has an $(\alpha, \beta)$-Conolly
  solution sequence.  We discover that Conolly-like sequences often
  appear as solutions to nested (or meta-Fibonacci) recurrence
  relations of the form $A(n) = \sum_{i=1}^k A(n-s_i-\sum_{j=1}^{p_i}
  A(n-a_{ij}))$ with appropriate initial conditions.  For any fixed
  integers $k$ and $p_1,p_2,\ldots, p_k$ we prove that there are only
  finitely many pairs $(\alpha, \beta)$ for which $A(n)$ can be
  $(\alpha, \beta)$-Conolly. For the case where $\alpha =0$ and $\beta
  =1$, we provide a bijective proof using labelled infinite
  trees to show that, in addition to the original Conolly recurrence,
  the recurrence $H(n)=H(n-H(n-2)) + H(n-3-H(n-5))$ also has the
  Conolly sequence as a solution. When $k=2$ and $p_1=p_2$, we
  construct an example of an $(\alpha,\beta)$-Conolly recursion for
  every possible ($\alpha,\beta)$ pair, thereby providing the first
  examples of nested recursions with
  $p_i>1$ whose solutions are completely understood.  Finally, in the case where $k=2$ and $p_1=p_2$, we provide an
  if and only if condition for a given nested recurrence $A(n)$ to be
  $(\alpha,0)$-Conolly by proving a very general ceiling function identity.
\end{abstract}


\maketitle

\section{Introduction}

In this paper we analyze recurrence relations of the form
\begin{align}
R(n) = \sum_{i=1}^k R\left(n-s_i-\sum_{j=1}^{p_i} R(n-a_{ij})\right)
\label{eq:generalRecurrence}
\end{align}
where the parameters $k$, $s_i$, $p_i$ and $a_{ij}$ are positive
integer constants (with the exception of $s_i$, which can be equal to
$0$).  We assume $c$ initial conditions $R(1) = \xi_1,R(2)=\xi_2,
\ldots, R(c)=\xi_c$, where unless otherwise indicated, $\xi_i> 0$.
\eref{generalRecurrence}~has a well-defined solution sequence as long
as for each $i$ in $\{1,2,\ldots,k\}$, the argument
$n-s_i-\sum_{j=1}^{p_i} R(n-a_{ij})$ is positive.

We use the following notation for \eref{generalRecurrence} and
optionally its initial conditions,
\begin{align}
\label{eq:generalNotation}
\langle s_1;a_{11},a_{12},\ldots,a_{1p_1} : s_2;a_{21},a_{22},\ldots,a_{2p_2} :
  \cdots :s_k;a_{k1},a_{k2},\ldots,a_{kp_k} \rangle [\xi_1,\xi_2,\ldots,\xi_c ].
\end{align}
For convenience we use this notation to refer to both the recursion
and its solution sequence even in the situation where we are uncertain
that a solution exists.  For example Hofstadter's $Q$-sequence
(\cite{Hofstadter}), defined by $Q(1) = Q(2) = 1$ and $Q(n) =
Q(n-Q(n-1)) + Q(n-Q(n-2))$ for $n >2$ is written $Q =
\SEQ{0}{1}{0}{2}[1,1]$; $Q$ is a famous example where it is not
known whether or not a solution exists.

We classify recurrences of the form ($\ref{eq:generalRecurrence}$)
according to the values of the parameters $p_i$ and $k$. We say the
above recursion has \emph{order} $\textbf{p} = (p_1,p_2,...,p_k)$, and for convenience we say it has order $p$ if $p_1 = p_2 = \cdots = p_k =p$.
We refer to $k$ as the \emph{arity} of a recursion. The case $k=1$
is uninteresting and yields only easily-understood periodic sequences (see \cite{Golomb1990} for an outline
of the required argument in a special case), so we only consider
recursions of arity at least 2.

%
%
%
%

To date, recursions of the form ($\ref{eq:generalRecurrence}$) have been examined only for $p=1$ (see, for example, \cite{AllenbySmith, CCT, HiTan, JacksonRuskey, Stoll, Tanny92}). As with nested recursions of order 1, no general solution methodology is available for general order $p$. Therefore a natural starting point for our analysis is the gathering and organization of experimental data
in a variety of situations.

Our empirical investigations focused on $k=2$ and $p \geq 1$, ie, recursions of the form

\begin{align}
  R(n) = R( n-s-R(n-a_1)\cdots-R(n-a_p) ) + R(
  n-t-R(n-b_1)\cdots-R(n-b_p)).
\label{eq:general2}
\end{align}

When it is convenient, we assume without loss of generality that $0 \le s \le t$
and $1 \le a_1 \le\cdots \le a_p$, and $1 \le b_1 \le \cdots \le
b_p$. Using the notation introduced in (\ref{eq:generalNotation}), we
write $R = \SEQ{s}{a_1,\ldots,a_p}{t}{b_1,\ldots,b_p}$.

Our empirical work indicated that many
combinations of parameters and initial conditions yield a solution
sequence reminiscent of the well-known Conolly sequence
\cite{ConVa}. This sequence is the solution to the \emph{nested} (also called
self-referencing or meta-Fibonacci) recurrence relation $C(n) =
\SEQ{0}{1}{1}{2}[1,2]$.  We say it is \emph{slowly growing} or \emph{slow} because it
has the property that successive differences are either 0 or 1 and it tends to infinity.  Any
slowly growing sequence can be described by its \emph{frequency
  sequence}, $\phi_C(m)$, that counts the number of times that $m$
occurs in $C(n)$. It is shown in \cite{JacksonRuskey} that $\phi_C(m)$
equals $r_m$, the so-called ``ruler function''. The ruler function
$r_m$ is defined as one plus the 2-adic valuation of $m$ (the exponent
of 2 in the prime factorization of $m$).  Thus $r_m =
1,2,1,3,1,2,1,4,1,2,1,3,1,2,1,5,\ldots$.

We observed that the aforementioned solution sequences related to the Conolly sequence have frequency sequences of the form $\alpha +\beta r_m$, for integers $\alpha$ and $\beta$; we call such sequences \emph{Conolly-like}. When we want to specify $\alpha$ and $\beta$, we say the sequence is the \emph{$(\alpha,\beta)$-Conolly} sequence. Correspondingly, we say that a recurrence relation is Conolly-like if it has a Conolly-like solution sequence and is \emph{$(\alpha,\beta)$-Conolly} when the solution is $(\alpha,\beta)$-Conolly.\footnote{Note that a recurrence is Conolly-like if there is at least one set of initial conditions for which its solution is Conolly-like; in general, for an arbitrary set of initial conditions, the solution, if it exists, will not be Conolly-like.} In this paper, we exclude from consideration the special case where $-\alpha = \beta > 0$, as the resulting sequence would not be slow.

Because of their close connection to the ruler function, Conolly-like sequences have many useful properties. In particular, following the approach in \cite{JacksonRuskey, FRusCDeg} it is straightforward to show that $\frac{z}{1-z} \prod_{n \ge 0} \left( 1+z^{2^n\alpha+(2^{n+1}-1)\beta} \right)$ is the generating function for the general $(\alpha,\beta)$-Conolly sequence\footnote{The central idea of this technique is to find a recursive description of the successive differences
  of the $(\alpha,\beta)$-Conolly sequence (these differences are either 0 or 1 because it is
  slowly growing), and then find the generating function of these differences.
}. Thus, proving that a given recursion is $(\alpha,\beta)$-Conolly also determines the generating function for its Conolly-like solution.

In the balance of this paper we focus on investigating conditions for the existence of slow Conlly-like solutions to recursions of the form ($\ref{eq:general2}$). To do so we
develop a proof technique for general $p$ that adapts the tree-based
combinatorial interpretations invented in \cite{JacksonRuskey,
  Rpaper, FRusCDeg} for the case $p=1$. Our idea is to build upon these
combinatorial interpretations, which identify an infinite tree associated with the Conolly
recursion and the recursion $\SEQ{0}{1}{2}{3}[1,1,2]$, respectively. These order
1 recursions have solutions with frequency sequences $r_m$ and 2,
respectively. For higher order Conolly-like recurrences our rough idea
is to create new trees as ``linear combinations" of the original two
trees.

Before we can apply the above approach effectively we must narrow the range of possible recursions of the form (\ref{eq:general2}) that are candidates for having Conolly-like solutions. We do so in \sref{Exp}~ where we explore asymptotic properties of any solution to the more general recurrence (\ref{eq:generalRecurrence}) ( see Theorems \ref{thm:ordlimit} and \ref{thm:NonSlowWeave}). By comparing these properties with the known asymptotic properties of an $(\alpha,\beta)$-Conolly sequence, we identify, for any given order $p$,
necessary restrictions on the parameters $\alpha$ and $\beta$ for any ($\alpha,\beta$)-Conolly solution to recursion ($\ref{eq:general2}$) (see Corollary \ref{cor:abConditions}). These findings provide important guidance for the
empirical search procedures that we adopt and motivate many of the results in the remainder of the paper.

Since our focus is on Conolly-like sequences, it
is natural to ask if the Conolly recursion
$\SEQ{0}{1}{1}{2}[1,2]$ is the only recursion of the form (\ref{eq:general2}) whose solution is the usual Conolly sequence (which of course is (0,1)-Conolly). From \sref{Exp}~ (see Table \ref{tbl:Pairs}) we know that any such recursion must be order 1; hence our focus on order 1 recursions in \sref{Order1}. There we use the tree technique outlined above to prove that $\SEQ{0}{2}{3}{5}[1,2,2,3,4]$ is another order 1 recursion whose solution is the Conolly sequence. Based on
our experimental results we conjecture that there are no other 2-ary
order 1  $(0,1)$-Conolly recursions. We conclude this section by describing completely all 2-ary order 1  $(2, 0)$-Conolly recursions (see Corollary \ref{cor:ceiling1}), which we know (again from Table \ref{tbl:Pairs} in \sref{Exp}) are the only other possibility for Conolly-like solutions to 2-ary recursions of order 1. This gives us a near-complete description of Conolly-like solutions for (\ref{eq:general2}) with $p=1$.

From Sections \ref{sec:Exp} and \ref{sec:Order1} we have necessary and
sufficient conditions for the existence of a $2$-ary, order $1$
$(\alpha,\beta)$-Conolly recursion. In \sref{AnyOrder} we extend this result to recursions of all orders (see Theorem \ref{thm:alphabeta}). To do so we apply a similar tree technique, given any $p,\alpha,\beta$ satisfying the necessary condition in \sref{Exp}, to construct an order $p$ ($\alpha,\beta$)-Conolly
recursion of the form ($\ref{eq:general2}$). This proves that
these conditions are also sufficient. We conclude this section by showing how solutions to low-order recursions can be used to identify solutions to recursions of higher order (see Theorems \ref{thm:interleave} and \ref{thm:perturb}).

In Section \ref{sec:AnyOrder} on recursions with $p>1$ we focused on finding examples of $(\alpha, \beta)$-Conolly recursions. In \sref{Ceiling} we address for $p>1$ the same question we (mostly) answered in Section \ref{sec:Order1} for $p=1$: for a given $(\alpha,\beta)$, what is the
complete list of $(\alpha, \beta)$-Conolly recursions of the form (\ref{eq:general2})? For $\beta >0$ we believe that for each possible $(\alpha, \beta)$ pair there are finitely many (and at least one) 2-ary $(\alpha, \beta)$-Conolly recursions; as an example we illustrate what we conjecture is a complete list of Conolly-like recursions that we determined empirically in the case $p=2$. For $\beta=0$ and arbitrary order $p$ we provide a complete answer (Theorem \ref{thm:ceiling}).

In \sref{Conc} we conclude with a discussion of possible directions for future work.

\section{Narrowing the Search for Conolly-Like Recursions} \label{sec:Exp}

Equation (\ref{eq:general2}) describes a very large class of recursions. Furthermore, it is not clear for which, if any, of the infinitely many pairs $(\alpha, \beta)$ a given recursion might be $(\alpha, \beta)$-Conolly. However, specifying that a sequence is  $(\alpha, \beta)$-Conolly implies very strong conditions on its asymptotic behavior determined by $\alpha$ and $\beta$. At the same time any sequence that solves (\ref{eq:general2}) must have asymptotic properties that depend upon the value of $p$. Therefore a Conolly-like solution to (\ref{eq:general2}) must satisfy both sets of requirements, necessitating a relationship between $p$ and $(\alpha, \beta)$. In what follows we derive this relationship; in doing so, because it is no more difficult, we establish a more general result that applies to all recursions (\ref{eq:generalRecurrence}).

The first step is to determine, given a limit $L$, what parameters for
($\ref{eq:generalRecurrence}$) can lead to a solution $A(n)$
such that the ratio $A(n)/n$ converges to $L$.

\begin{theorem}
  Let $A(n)$ be the solution sequence of a nested recursion of
  the form ($\ref{eq:generalRecurrence}$). Suppose that the ratio $A(n)/n$ has a nonzero limit $L$. Then $L
  = \frac{k-1}{\sum_{i=1}^{k}p_i}$. If $p_i = p$ for all $i$ and $k=2$, which is
  the focus of our attention in this paper, then $L = \frac{1}{2p}$.
\label{thm:ordlimit}
\end{theorem}

Note that the hypothesis $L >0$ in Theorem
$\ref{thm:ordlimit}$ rules out the trivial zero sequence solution.

\begin{proof}
  Let $A(n)$ satisfy the recursion $A(n) = \sum_{i=1}^k
  A\left(n-s_i-\sum_{j=1}^{p_i} A(n-a_{ij})\right)$. As a first step, we will show that we cannot have $p_iL > 1$ for any $i$. Assume to the
  contrary that for some $h$, and for some $\epsilon > 0$, $p_hL>1+\epsilon$. Then find some $N$ such that for all $j$ and for all $n>N$, $A(n-a_{hj}) > (1+\epsilon/2)(n-a_{hj})/p_h > n/p_h + n\frac{\epsilon}{2p_h} -2a_{hj}/p_h$. Then for $n>N$ and $n>\max\{\frac{4a_{hj}}{\epsilon}\}$, $\sum_{j=1}^{p_h} A(n-a_{hj}) > n \geq n-s_h$. But then the argument
  of the $h^{th}$ summand is negative, which is impossible since
  $A(n)$ is a solution. Thus, for all $i$, $p_iL \leq 1$. We will later be able to show that, in fact, $p_iL < 1$ for all $i$.

  Suppose there exists at least one value of $i$ for which $p_iL =
  1$. Without loss of generality let $p_1L=1$. Since $p_1 \ge 1$, we have  $\frac{A(n)}{n} \rightarrow L \leq 1$, so there exists some positive constant $\lambda$ such that for all $n$, $A(n) \leq \lambda n$. Then for all $n$,  $A(n-s_1-\sum_{j=1}^{p_1}A(n-a_{1j})) \leq \lambda (n-s_1-\sum_{j=1}^{p_1}A(n-a_{1j}))$. It follows that $\frac{A(n-s_1-\sum_{j=1}^{p_1}A(n-a_{1j}))}{n} \leq \lambda - \lambda s_1/n-\lambda \sum_{j=1}^{p_1}\frac{A(n-a_{1j})}{n}$. But $\frac{n}{n-a_{1j}} \rightarrow 1$, so $\frac{A(n-a_{1j})}{n} \rightarrow L$, from which we deduce that $\lim_{n \rightarrow \infty} \frac{A(n-s_1-\sum_{j=1}^{p_1}A(n-a_{1j}))}{n} \leq \lambda -\lambda \sum_{j=1}^{p_1}L = \lambda - \lambda p_1 L = 0$.

  Now, write
  \begin{align*}
    \frac{A(n)}{n} =
  \frac{A(n-s_1-\sum_{j=1}^{p_1}A(n-a_{1j}))}{n} + \frac{\sum_{i=2}^k
    A\left(n-s_i-\sum_{j=1}^{p_i} A(n-a_{ij})\right)}{n}
  \end{align*}
  and take the limit as $n \rightarrow \infty$ on both sides. The
  first term on the right vanishes. This argument shows that in
  evaluating $L$ we can ignore any summands with index $i$ such that
  $p_iL = 1$. Further, it confirms that we cannot have $p_iL = 1$ for
  all $i$, since then $L=0$, which contradicts our assumption.

  Combining the above results we may safely assume that $p_iL <1$ for
  all $i$. For each $i$ define $\kappa_i(n) := n-s_i-\sum_{j=1}^{p_i}
  A(n-a_{ij})$. Then $\lim_{n \rightarrow \infty}
  \frac{\kappa_i(n)}{n} = 1-p_iL > 0$, and thus $\lim_{n
    \rightarrow \infty} \kappa_i(n) = \infty$. But since $\lim_{n
    \rightarrow \infty} \frac{A(n)}{n} = L$ it follows that for all
  $i$,
  \begin{align*}\frac{A\left(n-s_i-\sum_{j=1}^{p_i}
        A(n-a_{ij})\right)}{n-s_i-\sum_{j=1}^{p_i} A(n-a_{ij})}
    \end{align*}
  converges to $L$ as $n \rightarrow \infty$.

  We can write
  \begin{align*}
    \frac{A(n)}{n} =&
  \sum_{i=1}^k\frac{A\left(n-s_i-\sum_{j=1}^{p_i}
      A(n-a_{ij})\right)}{n} \\
  =& \sum_{i=1}^k
  \left(\frac{n-s_i-\sum_{j=1}^{p_i} A(n-a_{ij})}{n}
  \right)
  \left(\frac{A\left(n-s_i-\sum_{j=1}^{p_i}
      A(n-a_{ij})\right)}{n-s_i-\sum_{j=1}^{p_i}
    A(n-a_{ij})}\right).
\end{align*}
Taking the limit on both sides as $n \rightarrow
  \infty$ we get $L = \sum_{i=1}^k (1 - p_iL)(L)$, from which we
  conclude (since $L>0$) that $1 = k-L\sum_{i=1}^kp_i$, or $L =
  \frac{k-1}{\sum_{i=1}^kp_i}$.

\end{proof}

In \tref{ordlimit} we assume \emph{a priori} the existence of the limit
$L$. This assumption is needed, a fact we explain further at the end of this section.

We now apply Theorem $\ref{thm:ordlimit}$ to Conolly-like meta-Fibonacci sequences.

\begin{corollary}
  Let $A(n)$ be a slow $(\alpha,\beta)$-Conolly sequence satisfying a
  recursion of the form (\ref{eq:general2}). Then $\alpha + 2\beta =
  2p$, $\alpha + \beta > 0$, and $\beta\geq 0$.
\label{cor:abConditions}
\end{corollary}

\begin{proof}
  Any slow sequence $A(n)$ has a positive frequency sequence,
  $\phi_A(m)$ (except possibly for some initial 0s). But $r_m = 1$ whenever $m$ is odd, so to have
  $\phi_A(m) = \alpha + \beta r_m > 0$ for any odd values of $m$, we must have $\alpha + \beta >
  0$. Because $r_m$ is unbounded
  and $\alpha$ is fixed, we must have $\beta \geq 0$ to ensure that
  $\phi_A(m) = \alpha + \beta r_m$ does not assume negative values. 

  The sequence $A(n)$ is $(\alpha,\beta)$-Conolly so has frequency
  sequence $\alpha+\beta r_m$. Consider the subsequence $n_h$ where
  $n_h$ is the largest integer with $A(n_h) = h$.  Then note that $n_h = h \alpha + \beta (2h + O(log(h)))$. This is because $r_1+r_2+\ldots+r_h = 2h + b_h$ where $b_h$ is the number of ones in the binary representation of $h$ (see \cite{ShallitAllouche} Theorem 3.2.1). Thus, $\frac{A(n_h)}{n_h}$ = $\frac{h}{h(\alpha+2\beta)+\beta O(log(h))}$ which converges to $\frac{1}{(\alpha + 2 \beta)}$. 

Next, we show that $\frac{A(n)}{n}$ and $\frac{A(n_h)}{n_h}$ converge to the same limit (recall that we do not \textit{a priori} know that $\frac{A(n)}{n}$ converges at all). Begin by noting that, since the sequence $A(n)$ has frequency sequence $\alpha + \beta r_m$, the number of terms between the last $h$ and the last $h+1$ will be $\alpha+\beta r_{h+1},$ that is, $n_h+\alpha+\beta r_{h+1} = n_{h+1}$. 

If $n$ is between $n_h$ and $n_{h+1}$, let $n=n_h+d$ (and note $0< d < \alpha+\beta r_{h+1}$). Then $A(n)/n = (h+1)/(n_h+d) = \frac{h}{n_h+d} + \frac{1}{n_h+d}$. Since $n_h$ is at least larger than $h$, and $d$ is bounded by a constant plus a constant times the ruler (which is in turn bounded by $1+log_2(h)$), we have that $d = O(log(n_h))$. Therefore, as $n$ increases, $\frac{h}{n_h+d}$ becomes arbitrarily close to $\frac{h}{n_h} = \frac{A(n_h)}{n_h}$, and $\frac{1}{n_h+d}$ vanishes. 

Hence $\lim_{n \rightarrow \infty} \frac{A(n)}{n} = \lim_{h \rightarrow \infty} \frac{A(n_h)}{n_h} = 
  \frac{1}{\alpha + 2 \beta}$. Thus by Theorem \ref{thm:ordlimit},
  $\alpha + 2 \beta = 2p$.
\end{proof}

From \cref{abConditions}~it follows that for given order $p$ there are $2p$ possible pairs $(\alpha,\beta)$, one for each of $\beta=0,1, \ldots, 2p-1$. For $p = 1,2,3,4$ these are given in Table \ref{tbl:Pairs}; recall that every $(\alpha,\beta)$-pair uniquely defines the corresponding Conolly-like sequence.

\medskip
\begin{table}[!ht]
\fontsize{10}{10}\selectfont
\caption{Possible $(\alpha,\beta)$-pairs for orders 1 to 4}\label{tbl:Pairs}
\center{
\begin{tabular}{c|l}
Order & Possible Pairs $(\alpha,\beta)$ \\ \hline
1 & $(2,0), (0,1)$ \\
2 & $(4,0), (2,1), (0,2), (-2,3)$ \\
3 & $(6,0), (4,1), (2,2), (0,3), (-2,4), (-4,5)$ \\
4 & $(8,0), (6,1), (4,2), (2,3), (0,4), (-2,5), (-4,6), (-6,7)$ \\
\end{tabular}
}
\end{table}
\medskip

For any fixed $p$ we now have a list of the finitely many ($\alpha, \beta$)-Conolly
sequences that could solve an order $p$ recurrence. However, so far we have only necessary conditions: for $p>1$ we haven't yet shown how to characterize which of the recursions (\ref{eq:general2}) are ($\alpha, \beta$)-Conolly.

Toward this end we wrote a program (in \texttt{C}) to systematically and efficiently
test restricted parts of the parameter space. Our search process is best described in the context of order 2, where there are only 4 possible $(\alpha,\beta)$ pairs, namely $(4,0)$, $(2,1)$, $(0,2)$,
and $(-2,3)$, and 6 parameters: $\SEQ{s}{a,b}{t}{c,d}$.

First we fixed an $(\alpha,\beta)$ pair. Then we explored parameter sets
that satisfied $0= s \le t \le 10$, $1 \le a \le b \le 12$, and $1
\le c \le d \le 30$.\footnote{Earlier experimental probing suggested that $s>0$ yielded no Conolly-like recursions.} When $s = t$ we only examined those with pairs
$(a,b)$ that are lexicographically less than or equal to $(c,d)$, so
as to avoid duplication.

With each parameter set we supplied a sufficient number of
initial conditions in order to seed the
recurrence. These values were the first 20 terms of the $(\alpha,\beta)$-Conolly sequence. We generated the first 1000 values for each sequence that was well-defined to that point and compared the result to the intended $(\alpha,\beta)$-Conolly sequence.

From our experimental data we discovered that for each of the four 2-ary order 2 $(\alpha,\beta)$-pairs indicated in Table
\ref{tbl:Pairs} there appear to be multiple Conolly-like recursions, including the recursion $\SEQ{0}{1,3}{\gamma}{\gamma+1,\gamma+3}$
where $\gamma = \alpha +\beta$. Based on this experimental evidence we conjectured that an analogous result is true
for 2-ary recursions of any order. This is the content of Theorem
$\ref{thm:alphabeta}$ below. Together with Corollary
$\ref{cor:abConditions}$ above this establishes necessary and sufficient
conditions on $\alpha$ and $\beta $ for the existence of at least one 2-ary order
$p$ $(\alpha,\beta )$-Conolly recursion.

Based on our experimental evidence, we conjectured that only finitely many $k$-ary order $p$ $(\alpha,\beta )$-Conolly recursions exist for fixed $k$ and $p$, and $\beta>0$. For $\beta=0$ we show in \sref{Ceiling} that there are infinitely many recursions of the form (\ref{eq:general2}); in fact, we demonstrate that for fixed $\alpha,k,p_i$ there are either no $(\alpha,0)$-Conolly recursions of the form (\ref{eq:generalRecurrence}) or infinitely many.

We close this section by addressing the issue we raised following Theorem \ref{thm:ordlimit}, namely, that there exist meta-Fibonacci sequences $A(n)$ that solve recursion ($\ref{eq:generalRecurrence}$) for which $A(n)/n$ does not
approach a limit. This demonstrates the necessity of the strict
conditions on Theorem $\ref{thm:ordlimit}$. One known example is Hofstadter's recursion $Q = \SEQ{0}{1}{0}{2}$ with alternate initial conditions $3,2,1$ \cite{Golomb1990}, the solution to which is the quasi-periodic sequence $Q(3w+1)=3, Q(3w+2)=3w+2$ and $Q(3w)=3w-2$.  The following result
illustrates how to construct such sequences for any values of $k$ and $p_i$. In
our approach we ``weave'' together two different solutions to the same
recursion.\footnote{Compare this result with Theorem \ref{thm:interleave} in \sref{AnyOrder}, where the resulting recursion has different order from the original one.}

\begin{theorem} (Fixed Order Interleaving Theorem)
  Let $A(n)$ and $B(n)$ both satisfy the same recursion of the form (\ref{eq:generalRecurrence}) with
  initial conditions $A(1),\ldots,A(r)$ and $B(1),\ldots,B(r)$
  respectively, so for $n>r$,
  \begin{align*}
    A(n) = \sum_{i=1}^k A\left(n-s_i-\sum_{j=1}^{p_i}
      A(n-a_{ij})\right) \text{ and } B(n) = \sum_{i=1}^k
    B\left(n-s_i-\sum_{j=1}^{p_i} B(n-a_{ij})\right).
  \end{align*}
  Define a sequence $C(n)$ as follows: $C(1),\ldots,C(2r) =
  2A(1),2B(1),2A(2),2B(2),\ldots,2A(r),2B(r)$ and for $n>r$, $C(2n-1) =
  2A(n)$ and $C(2n) = 2B(n)$. Then $C(n)$ is a solution of the ``doubled"
  recursion
  \begin{align*}
    \SEQ{2s_1}{2a_{11},2a_{12},\ldots,2a_{1p_1}}{
      2s_2;2a_{21},2a_{22},\ldots,2a_{2p_2} : \cdots
      :2s_k}{2a_{k1},2a_{k2},\ldots,2a_{kp_k}}
  \end{align*}
  of the form (\ref{eq:generalRecurrence}) with initial conditions $C(1),C(2),\ldots,C(2r)$.
\label{thm:NonSlowWeave}
\end{theorem}
\begin{proof}

  We show that for any $n > 2r$, $C(n) = \sum_{i=1}^k
  C\left(n-2s_i-\sum_{j=1}^{p_i} C(n-2a_{ij})\right)$. We proceed by
  induction for the case $n$ odd. The case $n$ even is entirely
  analogous so we omit the details.

  It is straightforward to verify that the result holds for
  $n=2r+1$. Assume that it holds for all odd $n$ up to $n=2m-3>2r$. We
  show it holds for $n = 2m-1$, where $m>r$.

  Applying the definition of $C(n)$ and simple algebra yields
  \begin{align*}
    &\sum_{i=1}^k C\left(2m-1-2s_i-\sum_{j=1}^{p_i}
      C(2m-1-2a_{ij})\right)
    = \sum_{i=1}^k C\left(2m-1-2s_i-\sum_{j=1}^{p_i} C(2(m-a_{ij}) - 1)\right)\\
    =& \sum_{i=1}^k C\left(2m-1-2s_i-\sum_{j=1}^{p_i}
      2A(m-a_{ij})\right)
    = \sum_{i=1}^k C\left(2(m-s_i-\sum_{j=1}^{p_i} A(m-a_{ij}))-1\right)\\
    =& \sum_{i=1}^k 2A\left(m-s_i-\sum_{j=1}^{p_i} A(m-a_{ij})\right)
    = 2 \sum_{i=1}^k A\left(m-s_i-\sum_{j=1}^{p_i} A(m-a_{ij})\right)
    = 2A(m).
  \end{align*}
  Note that we make use of the fact that $m-a_{ij}<m$ on several
  occasions in the above induction argument.

\end{proof}

It is evident that an analogous result can be proved in a similar way for three or more solutions and sets of initial conditions.

Using this result we construct a solution $C(n)$ to
a 2-ary order 1 recursion with the property that $\frac{C(n)}{n}$
does not have a limit; note that this technique generalizes to provide counterexamples of arbitrary order and arity. Let $A(n)$ satisfy the recursion $\langle
1;1:3;3 \rangle [0,0,0,0]$, which generates a sequence of all zeroes,
so $\frac{A(n)}{n}$ converges to 0. Let $B(n)$ satisfy the recursion
$\langle 1;1:3;3 \rangle [1,1,1,2]$, that is, the same recursion but
with different initial conditions. The solution $B(n)$ is the slow
sequence in which all numbers that are not powers of two appear twice,
and all numbers that are powers of two appear three times, so
$\frac{B(n)}{n}$ converges to $\frac{1}{2}$ (for a proof of these
facts about $B(n)$ see \cite{BLT}). Applying Theorem
$\ref{thm:NonSlowWeave}$ with the solutions $A(n)$ and $B(n)$ to the
recursion $\langle 1;1:3;3 \rangle$ results in the solution $C(n)$ to
the recursion $\langle 2;2:6;6 \rangle$. The even subsequence gives
$\lim \sup \frac{C(n)}{n} = \frac{1}{2}$ and the odd subsequence
yields $\lim \inf \frac{C(n)}{n} = 0$. Therefore the sequence
$\frac{C(n)}{n}$ has no limit.

\section{Order 1 Conolly-Like Recursions} \label{sec:Order1}

Recall from Table $\ref{tbl:Pairs}$ that the only Conolly-like sequences
that might solve a $2$-ary, order $1$ recursion of
the form ($\ref{eq:general2}$) are $(0,1)$-Conolly
and $(2,0)$-Conolly. We begin by examining which order 1 recursions are $(0,1)$-Conolly.

Aside from the original Conolly recursion $\SEQ{0}{1}{1}{2}[1,2]$ our detailed computer search identified only one candidate, namely, $\SEQ{0}{2}{3}{5}[1,2,2,3,4]$. We now prove that the Conolly sequence satisfies this recursion. We conjecture that this is the only other order 1 $(0,1)$-Conolly recursion.

\begin{theorem}
  Both of the order 1 recurrences
  \begin{align*}
    C(n)=C(n-C(n-1))+C(n-1-C(n-2))
  \end{align*}
  with $C(1) = 1, C(2) = 2$ and
  \begin{align*}
    H(n)=H(n-H(n-2)) + H(n-3-H(n-5))
  \end{align*}
  with $H(1) = 1, H(2) = H(3) = 2, H(4) = 3, H(5) = 4$
  have the Conolly sequence as their solution.
  \label{thm:L(n)Conolly}
\end{theorem}

As we pointed out in the Introduction, it is proved combinatorially in \cite{JacksonRuskey} that the ruler function $r_m$ is the
frequency sequence of the Conolly sequence by exhibiting an explicit
bijection between the Conolly sequence and the number of leaves of a
labeled infinite binary tree.  We elaborate on this methodology to
show that the ruler function is also the frequency sequence of $H(n)$.

We begin by defining the following infinite binary tree $T$.  We start with a sequence
of nodes called \emph{$s$-nodes} so that for each $m>0$, the $m^{th}$
$s$-node is the root of a complete binary tree of height $m+1$, with the
$(m-1)^{st}$ $s$-node as its left child. Call the nodes on the bottom level \emph{leaves}. Each leaf contains two cells, \cell{cell 1}{cell 2}. The nodes which are neither leaves nor $s$-nodes are called \emph{regular} nodes.

Define $T(n)$ to be $T$ with the $n$ labels $1,2,\ldots ,n$ inserted
in $T$ in the following way: label the nodes of $T$ in pre-order,
beginning with the leftmost leaf. Regular nodes receive one label,
$s$-nodes receive no labels, and leaves receive three consecutive
labels, the left cell receiving the first label and the right cell
receiving the following two labels. For example, if a left leaf child
has labels \cell{$a+1$}{$a+2,a+3$}, then its parent (at the
penultimate level) must be labeled with $a$ while its right sibling
has labels \cell{$a+4$}{$a+5,a+6$} so long as there are sufficient
labels (it may occur that there are insufficient labels to fully label
the final leaf). See Figure $\ref{fig:T(20)}$. A cell is considered
\emph{non-empty} if it contains at least one label, and non-labels are sometimes
denoted by $\varnothing$ to avoid ambiguity.


\begin{figure}
\begin{center}
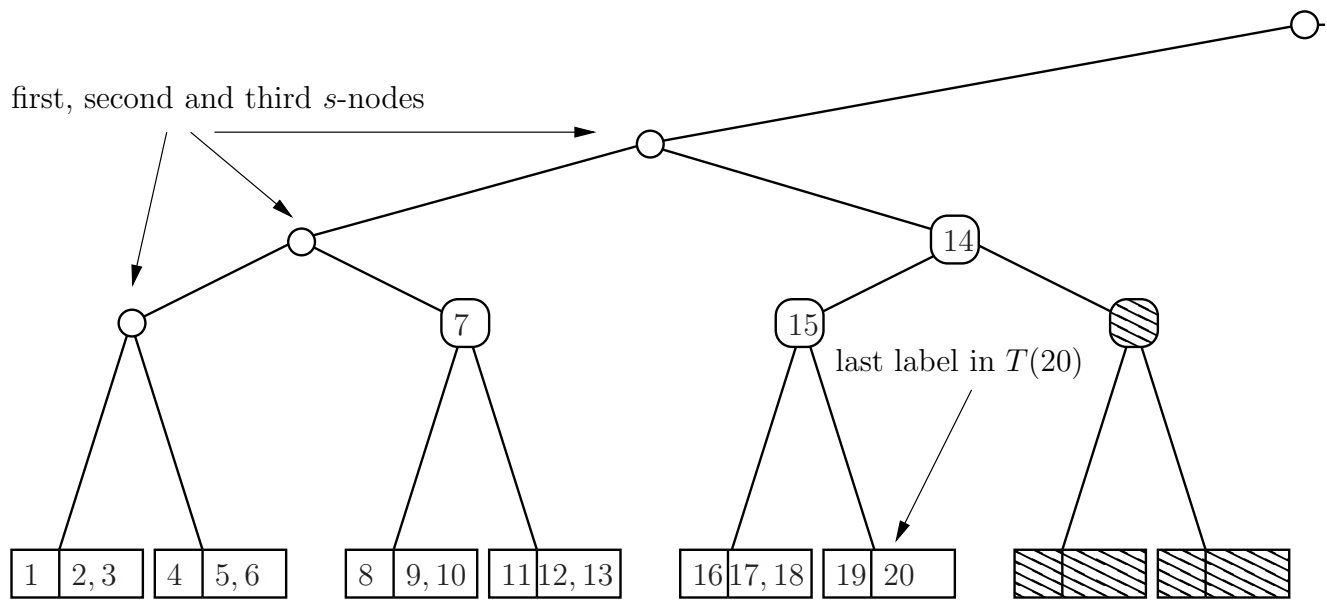
\caption{The infinite binary tree $T(20)$.}
\label{fig:T(20)}
\end{center}
\end{figure}

\begin{table}[!ht]
\fontsize{10}{10}\selectfont \caption{The sequence
$L(n)$ for $n =1,2,\ldots,20$}\label{tab:tabLn} \center{
\begin{tabular}{|c|c c c c c c c c c c c c c c c c c c c c|}
\hline
$n$                 & 1 & 2 & 3 & 4 & 5 & 6 & 7 & 8 & 9 & 10 & 11 & 12 & 13 & 14 & 15 & 16 & 17 & 18 & 19 & 20 \\
\hline
$L(n)$              & 1 & 2 & 2 & 3 & 4 & 4 & 4 & 5 & 6 & 6  & 7  & 8  & 8  & 8  & 8  & 9  & 10 & 10 & 11 & 12 \\
\hline
\end{tabular}
}
\end{table}

Let $L(n)$ be the number of non-empty cells in $T(n)$. See
\tabref{tabLn}~for the first twenty terms of the sequence $L(n)$,
which can readily be computed from $T(20)$ in
\fref{T(20)}.\footnote{It is evident that $T(m)$ is a sub-tree of
  $T(n)$, for $m\leq n$.}  Our strategy is to show that for all $n$,
$L(n) = H(n)$ and that $L(n)$ is the Conolly sequence.  Note that $L(n)$
has the same $5$ initial values as $H(n)$, namely, $1,2,2,3,4$. For $n>5$ we show that $L(n)$ satisfies the same recursion as $H(n)$, namely, $L(n)=L(n-L(n-2)) + L(n-3-L(n-5))$.


To do so we extend the general technique described in \cite{Rpaper}. We define a pruning operation that transforms the tree $T(n)$ with $n$ labels and $L(n)$ cells into the tree $T(n-L(n-2))$ with $n-L(n-2)$ labels and $L(n-L(n-2))$ cells. Before doing so we digress to establish a required technical result relating $L(n)$, the number of cells in $T(n)$, and $L(n-2)$. Our strategy is to use the fact that the tree $T(n-2)$ is obtained from $T(n)$ by removing the last two labels, $n-1$ and $n$.

\begin{lemma}
For all $n \ge 3$,
\[
L(n-2) =
\begin{cases}
L(n) & \mbox{ if } n \text{ is on a regular node,} \\
L(n)-1 & \mbox{ if } n \text{ is the first label on a leaf,} \\
L(n)-2 & \mbox{ if } n \text{ is the second label on a leaf,} \\
L(n)-1   & \text{ if } n \text{ is the third label on a leaf.}
\end{cases}
\]
\label{lemma:cellcount}
\end{lemma}
\begin{proof}
  If $n$ is on a regular node, then $n-1$ is either on a regular node or is the third label on a leaf (equivalently, the second label in the second cell of the leaf). Therefore, removing $n$ and $n-1$ cannot empty any cells so $L(n-2)=L(n)$. If $n$
  is the first label on a leaf, then again $n-1$ is either on a regular node or is the third label on a leaf. Thus removing $n$ and $n-1$ empties the cell containing $n$ but no others. Hence $L(n-2)=L(n)-1$. If $n$ is the second
  label on a leaf, i.e., the first label in the second cell of the leaf, then $n-1$ is the first label on that leaf and the only label in the first cell. Removing $n$ and $n-1$ empties both the cells in which they are contained, so $L(n-2) =L(n)-2$.  Finally, if $n$ is the third label on a leaf, then $n-1$ and $n$ are the two labels on that leaf's second cell, so removing these labels results in emptying only that cell. In this final case $L(n-2)= L(n)-1$.
%
\end{proof}

The use of the above lemma is as follows. We want to turn $T(n)$ into $T(n-L(n-2))$. The naive way to do this would be to simply remove the last $L(n-2)$ labels in preorder, but we cannot prove anything useful about the resulting tree with this pruning method. Rather, we exploit the fact that since $L(n)$ is the number of nonempty cells of $T(n)$, it follows that $T(n-L(n))$ is just $T(n)$ with one label deleted from every non-empty cell. We then use the above lemma to adjust $T(n-L(n))$ into $T(n-L(n-2))$, which together with other small changes preserves the original cell structure of $T$. In what follows we explain the precise details of this process.

\begin{figure}
\begin{center}

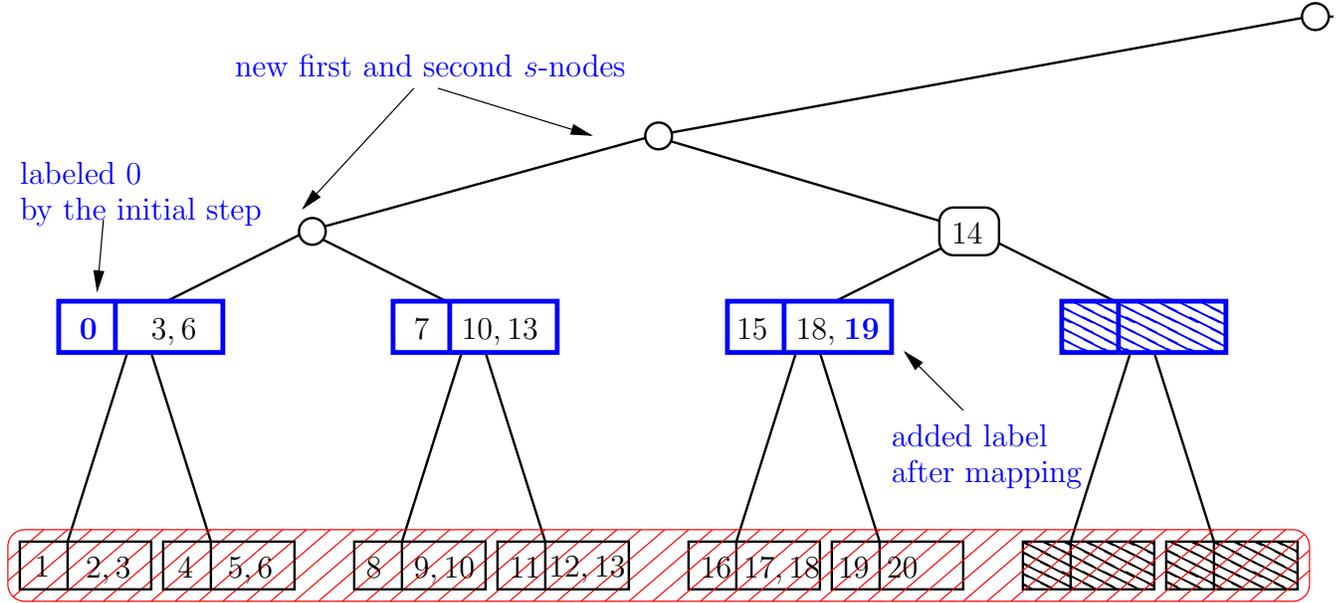

\caption{The pruning of $T(20)$ before the final relabeling. After
  relabeling this tree will be $T(10)$.}
\label{fig:pruning}
\end{center}
\end{figure}


We break down our pruning operation into several discrete steps that,
when combined, transform $T(n)$ to get to $T(n-L(n-2))$. See Figure
$\ref{fig:pruning}$ for an illustration with $n=20$. Begin by labeling
the first $s$-node with the label $0$, thereby making it a regular
node at the penultimate level (the \emph{initial step}). Next, delete
$L(n)$ labels from $T(n)$, any single label from each non-empty cell (the
\emph{deletion step}). In each of the penultimate level nodes of the
resulting tree, create two cells, inserting the existing label from
that node in the first of these cells (the \emph{cell creation
  step}). Empty all the leaves by moving any remaining leaf labels
(these labels must be in the second cell) to the second cell in their
respective parent at the penultimate level, and then delete all the
(empty) leaves (the \emph{lifting step}). Finally, following from
Lemma $\ref{lemma:cellcount}$, in some cases we need to add or
subtract one label to be sure we end up with exactly $n-L(n-2)$ labels
(recall that we have already added one label in the initial step). If
$n$ was on a regular node in $T(n)$, we remove the final label. If $n$
was the first or third label on a leaf in $T(n)$ we make no further
changes. If $n$ was the second label on a leaf in $T(n)$, we add one
label in the next available position in preorder (the \emph{correction
  step}).

We now check that the tree resulting from the pruning operation is $T(n-L(n-2))$. It is clear that the tree has the correct structure so we need only confirm that it contains the correct number of labels. Before the pruning operation $T(n)$ contained $n$ labels. After the initial step there are $n+1$ labels. Following the deletion step there are $n+1-L(n)$ labels. By Lemma $\ref{lemma:cellcount}$ the labels, if any, introduced or deleted in the correction step result in a total of $n-L(n-2)$ labels. Notice that we could renumber the labels of this tree from $1$ to $n-L(n-2)$ but since only the number of labels matters we omit this step.

\begin{theorem}
  If $n \ge 6$ then the number of nonempty cells in the left leaves of $T(n)$
  is $L(n-L(n-2))$ and the number of nonempty cells in the right leaves of
  $T(n)$ is $L(n-3-L(n-5))$. Hence $L(n)= L(n-L(n-2))+L(n-3-L(n-5))$.
\label{thm:leftright}
\end{theorem}
\begin{proof}
  We first show that the number of non-empty cells in the left leaves of $T(n)$ is
  $L(n-L(n-2))$. Consider an arbitrary nonempty penultimate level node $X$ of $T(n)$ with the label $a$. We show that after the pruning process is complete, $X$ will have the same number of nonempty cells as its left child had before the pruning process. We consider 5 cases.

  Case 1: $n=a$. In this case after the cell creation step in the pruning process $X$ will be labeled \cell{$a$}{$\varnothing$,$\varnothing$}. Since $X$ has empty children, after the lifting step, the new labeling on $X$ will remain unchanged. Since $X$ is the last non-empty node in preorder for the correction step, and since initially $n$ was on the regular node $X$, the correction step deletes the label $a$, leaving $X$ empty. Thus, after the pruning process $X$ has no non-empty cells, as required.

  Case 2: $n=a+1$. Here the left child of $X$ is labeled \cell{$a+1$}{$\varnothing$,$\varnothing$}; the right child is empty. After the lifting step, $X$ will be labeled \cell{$a$}{$\varnothing$,$\varnothing$}, and the correction step does nothing because $n$ was the first label on a leaf. Thus, after the pruning process $X$ has one nonempty cell, corresponding to the one non-empty cell that its left child had before the pruning process.

  Case 3: $n=a+2$. In this case the left child of $X$ is labeled \cell{$a+1$}{$a+2$,$\varnothing$} and the right child is empty. So, after the lifting step, $X$ will be labeled \cell{$a$}{$\varnothing$,$\varnothing$}. Since $X$ is now last in preorder, and $n$ was the second entry in a leaf, the correction step adds one label to $X$, causing it to have an entry in its second cell after the pruning process. Thus $X$ has two non-empty cells, corresponding to the two non-empty cells that its left child had before the pruning process.

  Case 4: $n=a+3,a+4$ or $a+5$. In all of these cases, the left child of $X$ is full, with labels \cell{$a+1$}{$a+2,a+3$} while the right child of $X$ has at most two labels so at most one label per cell. After the lifting step, $X$ is labeled \cell{$a$}{$a+3$,$\varnothing$}. Since the correction step will either do nothing or add a second label to the already nonempty second cell of $X$, $X$ ends the pruning process with two nonempty cells, as required.

  Case 5: $n \geq a+6$. In this final case, the left and right children of $X$ are fully labeled as \cell{$a+1$}{$a+2,a+3$} and \cell{$a+4$}{$a+5,a+6$}, respectively. Thus, after the lifting step, $X$ is labeled \cell{$a$}{$a+3,a+6$}. The correction step may remove at most one label from $X$, namely, the label $a+6$. But doing so does not empty the second cell of $X$, so $X$ ends the pruning process with two nonempty cells.

  We next show that the number of nonempty cells in the right leaves of $T(n)$ is
  $L(n-3-L(n-5))$. First note that by substituting $n-3$ for $n$ in
  the above discussion we conclude that $L(n-3-L(n-5))$ counts the number of
  nonempty cells in the left leaves of $T(n-3)$. Since leaves of $T$ can hold up to three labels, if we have a sibling pair consisting of a left and right leaf, then a cell in the left leaf will be nonempty in $T(n-3)$ if and only if the corresponding cell in the right leaf is nonempty in $T(n)$. Thus, the number of nonempty cells in the left leaves of $T(n-3)$ is the same as
  the number of nonempty cells in the right leaves of $T(n)$. This completes the proof.
\end{proof}

We now show that $L(n)$ is the Conolly sequence, from which it
follows that the frequency sequence for $L(n)$ is the ruler function.

\begin{theorem}
The solution sequence $L(n)$ of the recursion
\[
H(n)=H(n-H(n-2)) + H(n-3-H(n-5))
\]
with initial conditions $H(1) = 1, H(2) = H(3) = 2, H(4) = 3, H(5) = 4$ is the Conolly sequence.

\end{theorem}
\begin{proof}
  We prove that $L(n)$ is the Conolly sequence by showing that both
  $L(n)$ and the Conolly sequence have the same first difference sequences.
  Define the finite binary string $D_n$ recursively by the rules $D_0
  = 1$ and $D_{n+1} = 0D_n D_n$ whenever $n\geq 0$.  Then the infinite
  sequence of successive first differences $C(n+1)-C(n)$ of the
  Conolly sequence $C(n)$ is given by
\[
\mathcal{D} = D_0D_0D_1D_2D_3 \cdots,
\]
with the convention that $C(0) = 0$ so the difference sequence starts
with $C(1) - C(0) = 1$ (see \cite{JacksonRuskey}).

Let $d(n) = L(n) - L(n-1)$. Then $d(n)$ is the increase in the
number of non-empty cells when applying an $n$th label to $T(n-1)$,
where we take $T(0)$ to be the unlabeled tree $T$, that is, $T(0) = T$ and $L(0) = 0$. Define the binary string $F_n$ by the rules
$F_1 = 110110$, $F_2 = 0F_1$ and $F_{n+1} = 0F_nF_n$ whenever
$n\geq 2$. We show by induction that the infinite binary sequence
\[
\mathcal{F} = F_1F_2F_3 \cdots
\]
gives the sequence $d(1)d(2)d(3)\ldots$

To do this, we introduce a new symbol. For all $m>0$, let $T_m$ be the subtree of $T$ consisting of the right child of the $m^{th}$ $s$-node and all the descendants of that right child.

We first show that $F_1F_2 = d(1)d(2)\cdots d(13)$. The string $F_1$ gives the subsequence of $d(n)$ corresponding to any pair of leaf siblings in $T(n)$ (and hence $d(1)d(2)\cdots d(6) = F_1$).  The right child of the second $s$-node of $T$ is
the root of $T_2$, a $3$-node complete binary tree and clearly $d(7)d(8)\cdots d(13) = 0F_1 = F_2$, so $F_2$ is the subsequence of $d(n)$ corresponding to $T_2$.

We now show inductively that for $m\geq 2$, $F_m$ is the subsequence of $d(n)$ corresponding to $T_m$. We already have the base case, since $F_2$ is the subsequence of $d(n)$ corresponding to $T_2$. Assume that $F_{m-1}$ is the subsequence of $d(n)$ corresponding to $T_{m-1}$. The key observation is that for $m \geq 3$, $T_m$ consists of a single regular node from which descend two copies of $T_{m-1}$. The preorder traversal of $T_m$ begins with the regular node at its root (corresponding to a $0$ in the difference sequence) followed by the left copy of $T_{m-1}$ (corresponding to one repetition of $F_{m-1}$ in the difference sequence by the induction hypothesis) followed by the right copy of $T_{m-1}$ (hence another repetition of $F_{m-1}$ in the difference sequence). Therefore the part of the difference sequence corresponding to $T_m$ is $F_m = 0F_{m-1}F_{m-1}$. Since the sequence of subtrees $T_m$ are labeled successively, and the $s$-nodes that join them in pairs are empty, $d(1)d(2)\cdots = \mathcal F$.

Finally, we show that $\mathcal F = \mathcal D$, from which we
conclude that $L(n)$ matches the Conolly sequence. In what follows, by the notation $0^{-1}$ we mean the inverse of $0$ in the free group on the symbols $0,1$ which make up these binary strings. Similarly, $1^{2}$ means the string $11$, and so on. Observe that $F_1 =
D_0D_0D_10$. For $n>1$, we show by induction that $F_n = 0^{-1}D_n0$,
from which we get that $\mathcal F = \mathcal D$. For $n=2$ this can
be verified directly: $F_2 = 01^201^20 = 0^{-1}0^21^201^20 =
0^{-1}0D_1D_10 = 0^{-1}D_20$.

For $n\geq 2$ assume $F_n = 0^{-1}D_n0$. Then $F_{n+1}
= 0F_nF_n=00^{-1}D_n00^{-1}D_n0 = D_nD_n0 = 0^{-1}0D_nD_n0 =
0^{-1}D_{n+1}0$ as required. This completes the induction.
\end{proof}

%

Recall from Table $\ref{tbl:Pairs}$ that there are only two possibilities for the values of the $(\alpha, \beta)$ pairs for Conolly-like solutions of 2-ary, order 1 meta-Fibonacci recursions, namely, $(\alpha, \beta) = (0,1)$ or $(2,0)$. We conclude our discussion of such recursions by identifying all the  $(2, 0)$-Conolly 2-ary, order 1 meta-Fibonacci recursions. Note that the $(2,0)$-Conolly solution to these recursions is by definition the ceiling sequence $\cln{n}{2}$ (see \cite{BLT} for a special case). Our characterization of these recursions follows immediately from a more general result (see \tref{ceiling} below).

\begin{corollary}
  The function $\cln{n}{2}$ is the unique solution that satisfies the 2-ary, order 1
  recursion $R = \SEQ{s}{a}{t}{b}$, given sufficiently many initial conditions $R(z) = \cln{z}{2}$, if and only if $a$ and $b$ are both odd, and $2(s+t) = a+b$.
\label{cor:ceiling1}
\end{corollary}

For any given set of parameters $\set{s,a,t,b}$ with $a$ and $b$ both odd and $2(s+t) =
a+b$, it is easy to show that
\begin{align*}
m = \max\set{a,b,s+\frac{a+1}{2},t+\frac{b+1}{2}}
\end{align*}
initial conditions are sufficient. This upper bound is not tight,
however.  For example, if the parameters are 1,3,3,5 then the initial
conditions 1,1,2,2,3 are sufficient, but $m=6$ in
this case.



\section{Conolly-Like Recursions of Higher Order} \label{sec:AnyOrder}

From Sections \ref{sec:Exp} and \ref{sec:Order1} we have necessary and
sufficient conditions for the existence of a $2$-ary, order $1$
$(\alpha,\beta)$-Conolly recursion. In this section
we extend this result to recursions of all orders using a proof
strategy similar to that of \tref{leftright}. For this purpose we
invent a new labeling of the same tree $T$ defined in
\sref{Order1}.

First, we state our key result:

\begin{theorem}\label{thm:alphabeta}
  There exists a 2-ary  $(\alpha, \beta)$-Conolly meta-Fibonacci recursion if and only if $\beta \geq 0,
  \alpha+\beta >0$, and $\alpha$ is even.
\end{theorem}

Observe that the ``only if'' part of \tref{alphabeta} follows directly
from \cref{abConditions}. To prove sufficiency we claim that the
recursion

\begin{align}
  \SEQ{0}{1,3,5,\ldots,2p-1}{\gamma}{\gamma+1,\gamma+3,\gamma+5,\ldots,\gamma+2p-1}
  \label{eq:alphabetarecursion}
\end{align}
has an $(\alpha,\beta)$-Conolly solution, where $\gamma=\alpha+\beta$, and where it follows from Corollary $\ref{cor:abConditions}$ that $p$ must be $\alpha/2+\beta$.

To establish our claim we first describe the infinite binary tree $U$ that provides the required combinatorial interpretation for the solution to the above recursion, as well as the pruning process that we apply to $U$. The tree $U$ has the same general structure as the tree $T$ used in \tref{leftright}, that is, $U$ has $s$-nodes, regular nodes, and leaves in the same
locations as $T$. As with $T$, the $s$-nodes of $U$ are always empty. However, there are some important differences from $T$ in the way that we label the remaining nodes of $U$. The regular nodes of $U$ contain up to $\beta$ labels
(rather than $1$ label, as in $T$). The leaves of $U$ do not have multiple cells; each leaf of $U$ can contain a total of $\alpha+\beta$ labels (rather than a maximum of 3 labels in $T$).

Define $U(n)$ to be $U$ with the $n$ labels $1$ through $n$ assigned
in preorder. As was the case with $T(n)$, we refer to nodes of $U(n)$
that have no labels as ``empty," and nodes in $U(n)$ containing their
maximum number of labels as ``full." Only the last nonempty node in
$U(n)$, the one containing the label $n$, can be partly
filled. Analogous to the function $L(n)$ in \tref{leftright}, define
$M(n)$ to be the number of nonempty leaf nodes on $U(n)$. Unlike $L$,
since the leaves of $U(n)$ do not have multiple cells, $M$ counts
leaves, not cells. See $U(17)$ in \fref{uofn} where $\alpha = 2$ and
$\beta = 1$.

\begin{figure}
\begin{center}
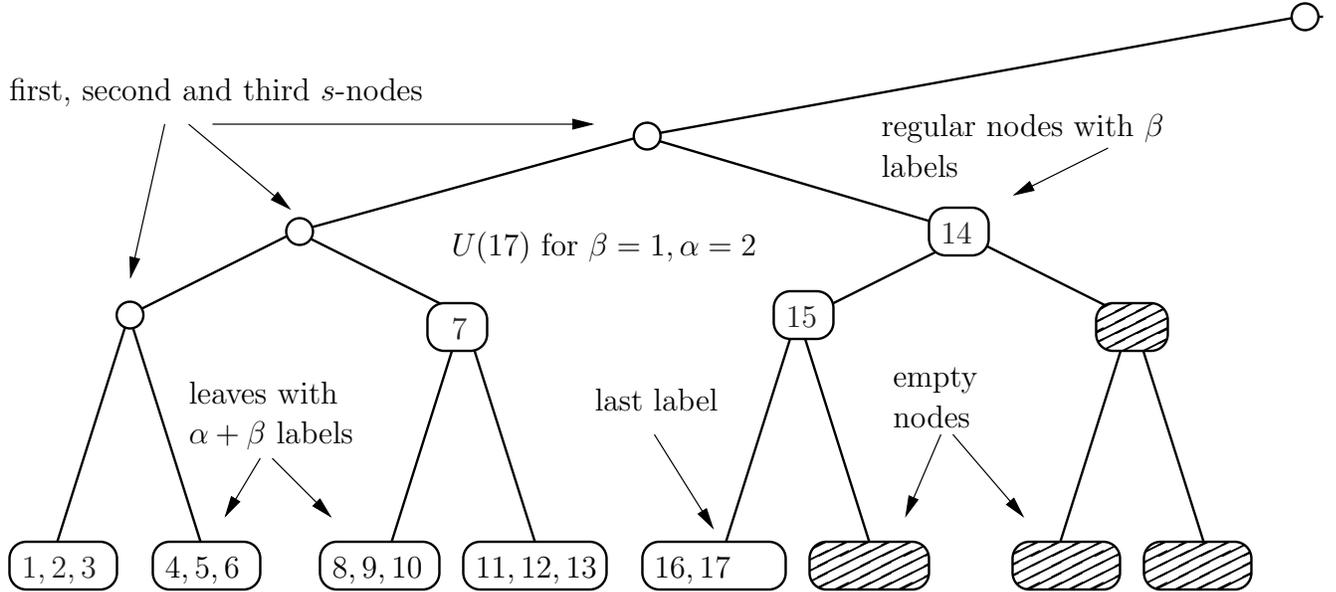
\caption{The labeled tree $U(17)$ with $\alpha = 2$ and $\beta = 1$.
We see that $M(17) = 5$.}
\label{fig:uofn}
\end{center}
\end{figure}

We define a pruning operation on the tree $U(n)$ which transforms
$U(n)$ into the tree $U(n-\sum_{j=1}^{j=p}M(n-2j+1))$. Because of the
much greater generality of the result that we seek to prove here this
pruning operation is considerably more complex than the one we
described above for $T(n)$. We proceed with the explanation in stages,
beginning with $\alpha \geq 0$.

The \textit{initial step} is as follows: take $U(n)$ and add $\beta$
labels to the first $s$-node, thereby making this node identical to
the other full regular nodes at the penultimate level (now the tree
has $n+\beta$ total labels). Next, the \textit{deletion step}: for
each leaf $Y$ of $U(n)$, remove one label from $Y$ for each of the
trees $U(n-1), U(n-3), \ldots, U(n-2p+1)$ in which $Y$ is not an empty
node. For example, if the first label on $Y$ is $a$, and $n \geq a +
2p - 1$, then $Y$ is necessarily full in $U(n)$ and nonempty in all
$p$ of the trees $U(n-1),U(n-3),\ldots,U(n-2p+1)$. Thus, in this case
the deletion step removes $p$ labels from $Y$. More generally, if $Y$
is any full leaf, there are certainly at least $p$ labels available to
remove since $\alpha \geq 0$ and $p = \alpha/2 + \beta \leq \alpha +
\beta$; if $Y$ is not full then by \lref{deletionstepcap} below we
will not delete all the labels of $Y$. Our tree now has
$n+\beta-\sum_{j=1}^{j=p}M(n-2j+1)$ labels.

Notice that, as a further consequence of  \lref{deletionstepcap}, each leaf (except any children of the last penultimate node) will have precisely $\alpha/2$ labels left (with the children of the last penultimate node having at most $\alpha/2$ labels left). We deal with all of these leftover labels in the \textit{lifting step}: we take any labels remaining in leaves and lift them into the parent of their leaf node in
$U$. Complete the lifting step by deleting all of the (now empty)
leaves. This means that every penultimate node except for the last penultimate node will recieve precisely $\alpha$ labels, and the last penultimate node will recieve at most $\alpha$ labels. Therefore, all penultimate nodes but the last will have exactly $\alpha+\beta$ total labels (since they started with $\beta$), and the last penultimate node will have at most $\alpha+\beta$ total labels.

We complete the pruning process with the \textit{correction step}:
delete the last $\beta$ labels (by preorder) in the tree. It is clear
from the construction that our final pruned tree has
$n-\sum_{j=1}^{j=p}M(n-2j+1)$ labels and is
$U(n-\sum_{j=1}^{j=p}M(n-2j+1))$, up to renumbering of labels. See
\fref{uofnprune}.

\begin{figure}
\begin{center}
 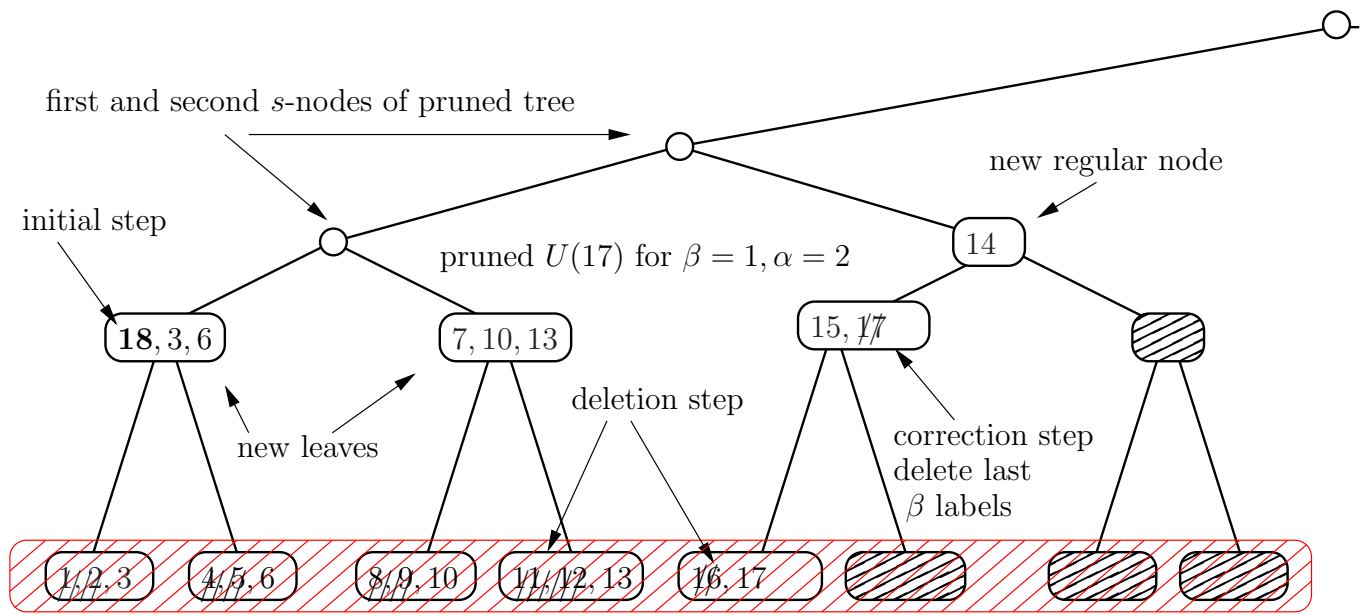
 \caption{The pruning of $U(17)$ with $\alpha = 2$ and $\beta = 1$
   before the final relabeling. After relabeling this tree will be
   $U(8)$.}
\label{fig:uofnprune}
\end{center}
\end{figure}

We now describe how to prune $U(n)$ when $\alpha < 0$. In this case
the leaves contain fewer than $p$ labels. We follow the same pruning
process as above, except when it would require that we delete more
labels from a leaf node than that node contains. If the leaf $Y$ has a
deficit $\delta$ in the number of available labels to delete, then we
delete all the labels from $Y$ and $\delta$ labels from the parent of
$Y$. Note that since $\delta \leq -\alpha/2$, we delete at most
$-\alpha$ labels from the parent of each leaf. Since $\alpha + \beta >
0$, we always have enough labels in the penultimate level nodes for
this to be possible. See Figures \ref{fig:uofnnegalpha} and
\ref{fig:uofnnegalphaprune}.

\begin{figure}
\begin{center}
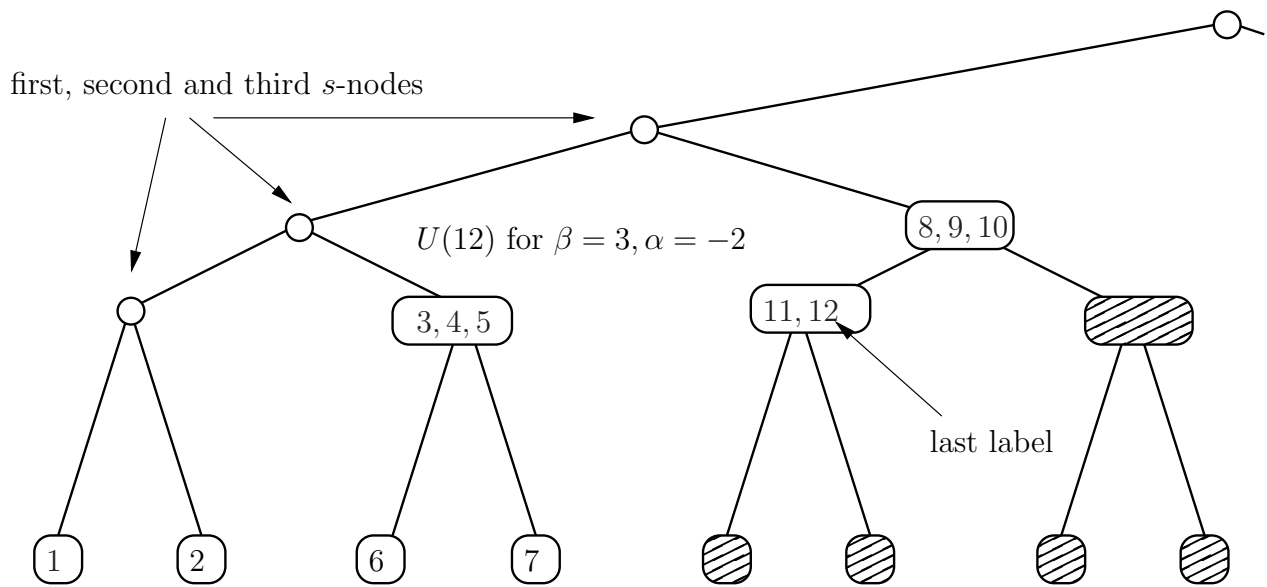
 \caption{The labeled tree $U(12)$ with $\alpha = -2$ and $\beta = 3$.}
\label{fig:uofnnegalpha}
\end{center}
\end{figure}

\begin{figure}
\begin{center}
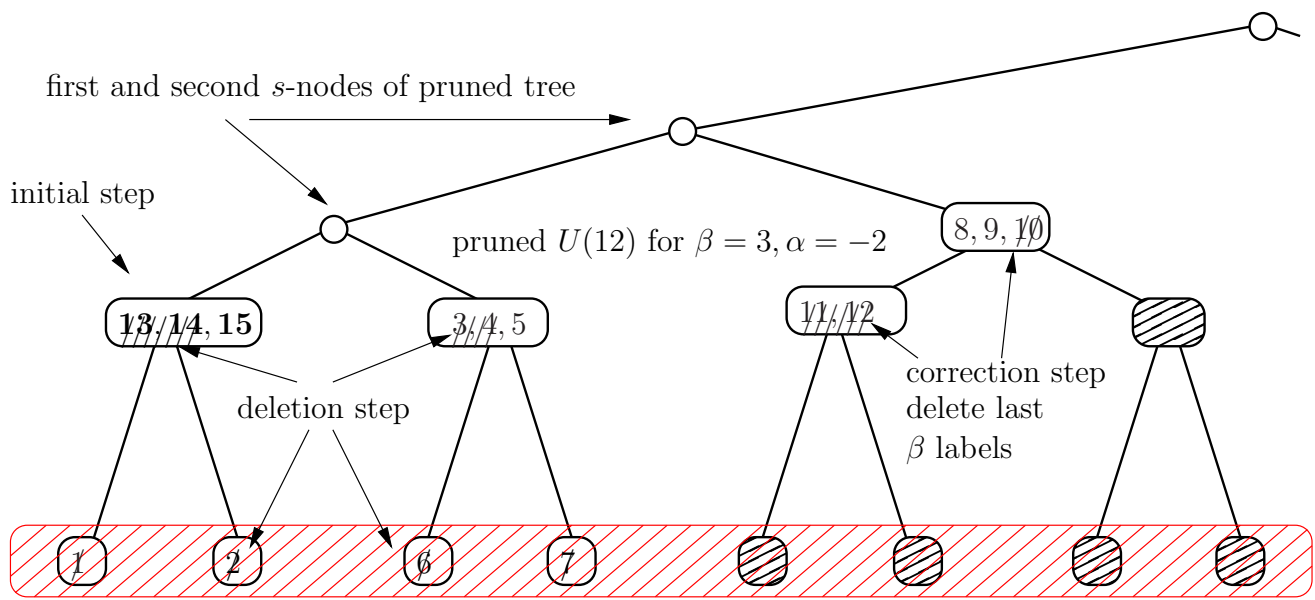
 \caption{The pruning of $U(12)$ with $\alpha = -2$ and $\beta = 3$
   before the final relabeling.  After relabeling this tree will be $U(4)$.}
\label{fig:uofnnegalphaprune}
\end{center}
\end{figure}

Analogous to Lemma $\ref{lemma:cellcount}$, we count the number of labels that are deleted from any given leaf in the following lemma.

\begin{lemma}
  \label{lem:deletionstepcap}

Let $Y$ be a leaf node of $U(n)$. Suppose there are $d$ labels on or
after $Y$ in preorder (that is, labels either on $Y$ or on nodes which are further along in preorder than $Y$). Then during the deletion step of the pruning
process, $min\{\lfloor d/2 \rfloor, p\}$ labels are deleted from $Y$.
\end{lemma}
\begin{proof}
  Clearly, if $\lfloor d/2 \rfloor > p$, then $p$ labels are deleted
  from $Y$ since $Y$ is nonempty on $U(n-2i+1)$ for $i \leq
  p$. Otherwise, since $2i-1 < d$ if and only if $i \leq \lfloor d/2
  \rfloor$, we will have at least one label on or after $Y$ for all
  the trees $U(n-2i+1)$ for $i \leq \lfloor d/2 \rfloor$. Thus $
  \lfloor d/2 \rfloor$ labels will be deleted from $Y$.
\end{proof}

We are now prepared to prove \tref{alphabeta}.

\begin{proof}
As discussed above, it suffices to show that ($\ref{eq:alphabetarecursion}$) has an $(\alpha, \beta)$-Conolly solution.

  First, we show that $M(n)$ is an $(\alpha,\beta)$-Conolly
  sequence. If $\alpha = 0, \beta = 1$, then $U$ is identical to the tree studied in
  \cite{JacksonRuskey}, where it is shown that $M(n)$ is the Conolly sequence. For $\alpha = 0$ and
  $\beta>1$, then every node holds $\beta$ times as many labels as in the tree in \cite{JacksonRuskey}, which multiplies the number of times each integer appears in the sequence $M(n)$ by $\beta$; the result is a
  $(0,\beta)$-Conolly sequence. Increasing or decreasing $\alpha$
  changes only the number of labels per leaf. This means that every integer
  will appear an additional $\alpha$ times (or $-\alpha$ fewer times),
  giving an $(\alpha, \beta)$-Conolly sequence.

Next, we show that
\begin{align*}
\SEQ{0}{1,3,5,\ldots,2p-1}{\gamma}{\gamma+1,\gamma+3,\gamma+5,\ldots,\gamma+2p-1},
\end{align*}
with initial conditions $M(1), \ldots, M(4\alpha + 5 \beta)$, is
$M(n)$. This will establish the theorem. Note that $4\alpha + 5 \beta$
is the last label on the $4^{th}$ leaf of $U$.

First we prove that for all $n>4\alpha + 5 \beta$,
$M(n-\sum_{j=1}^pM(n-2j+1))$ counts the number of left leaves of
$U(n)$.  Recall that the above definition of pruning transforms the
tree $U(n)$ into the tree $ U(n-\sum_{j=1}^{j=p}M(n-2j+1))$ (up to
renumbering of the labels). Thus, $M(n-\sum_{j=1}^pM(n-2j+1))$ counts
the number of nonempty leaves of $U(n-\sum_{j=1}^{j=p}M(n-2j+1))$. We
exhibit a bijection between the nonempty leaves of
$U(n-\sum_{j=1}^{j=p}M(n-2j+1))$ and the nonempty \textit{left} leaves
of $U(n)$; equivalently, we show that in the pruning process that
transforms $U(n)$ into $U(n-\sum_{j=1}^{j=p}M(n-2j+1))$, a given
penultimate level node $X$ of $U(n)$ will be nonempty in
$U(n-\sum_{j=1}^{j=p}M(n-2j+1))$ if and only if $X$ has a nonempty
left leaf child in $U(n)$. To do so we consider 4 cases corresponding
to the position of the last label $n$ on $U(n)$. For any penultimate
node $X$ we write $Y$ (respectively, $Z$) for the left (respectively,
right) leaf child of $X$.

Case 1: The label $n$ is the $d^{th}$ label on the left leaf child $Y$
of $X$. In this case, we need to prove that $X$ ends the pruning
process with at least one label. By Lemma $\ref{lem:deletionstepcap}$,
the deletion step will remove at most $\lfloor d/2 \rfloor < d$ labels
from $Y$, so $Y$ will end the deletion step with at least one
label. Thus, in the lifting step at least one label shifts up from $Y$
to $X$, causing $X$ to have more than $\beta$ labels after the lifting
step. The correction step of the pruning process deletes the last
$\beta$ labels in preorder, and these labels will be on $X$. Since $X$
now has more than $\beta$ labels, it will end the pruning process with
at least one label.

Case 2: The label $n$ is the $d^{th}$ label on the right leaf child
$Z$ of $X$. Again, we need to prove that $X$ ends the pruning process
with at least one label. Note that $Y$ must be full since $Z$ is
nonempty. As in Case 1, by Lemma $\ref{lem:deletionstepcap}$ the
deletion step removes at most $\lfloor d/2 \rfloor < d$ labels from
$Z$, so we do not delete every label of $Z$. Furthermore, since there
are $\alpha+\beta+d$ labels on or after $Y$ and $d \leq \alpha + \beta$ labels on $Z$, there are at most
$\lfloor (\alpha + \beta + d)/2 \rfloor \leq \alpha + \beta$ labels
deleted from $Y$ during the deletion step. Since there is no label
deficit on $Y$, no labels need to be deleted from $X$. Therefore,
before the lifting step, $X$ will have $\beta$ labels and $Z$ will
have at least one label. Thus, after the lifting step, $X$ will have
more than $\beta$ labels. The correction step of the pruning process
deletes the last $\beta$ labels in preorder, and these labels will be
on $X$. Since $X$ now has more than $\beta$ labels, it will end the
pruning process with at least one label.

Case 3: The label $n$ is on $X$. Since $X$ has an empty left child, we
want to show that $X$ ends the pruning process without any labels. But
$n$ is the last label in preorder, and since $X$ has no children it
will not have any labels removed during the deletion step or added
during the lifting step. Thus, the $\beta$ labels that are removed
from the tree in the correction step of the pruning process will come
out of $X$ first. But $X$ has at most $\beta$ labels. Therefore, after
the pruning process, $X$ will have no labels.

Case 4: The label $n$ is after $Z$ in preorder. We show that after the
pruning process $X$ retains at least one label. To do so we must
consider 3 subcases.

Subcase 4.1: There are $d$ labels on $U(n)$ after the last label of
$Z$, where $1 \leq d < \beta$. Since a regular node contains up to
$\beta$ labels, it follows that these $d$ labels must all be on the
regular node, call it $X'$, that follows $Z$ in preorder; note that
$X'$ has no non-empty children.

In this case, there are exactly $\alpha+\beta+d$ labels in preorder
that are on or after $Z$ and $2(\alpha + \beta)+d$ labels in preorder
that are on or after $Y$. By Lemma $\ref{lem:deletionstepcap}$, the
number of labels deleted during the deletion step from $Z$ (resp. $Y$)
is $\lfloor (\alpha+\beta+d)/2 \rfloor = \alpha/2 + \lfloor (\beta +
d)/2 \rfloor$ (resp. $\lfloor (2(\alpha + \beta)+d)/2 \rfloor = \alpha
+ \beta +\lfloor d/2 \rfloor$). Further, since $X'$ has no children,
$X'$ will have no labels added or deleted during the deletion step and
the lifting step. In the correction step, we remove the last $\beta$
labels from $X'$ and $X$ in our transformed tree. We now count labels
carefully: when the pruning process began $U(n)$ had $2 \alpha + 3
\beta + d$ labels on or after $X$ (including the $\beta$ labels on
$X$).

From the calculations in the above paragraph, during the the deletion and correction steps we deleted $(\alpha/2 + \lfloor (\beta + d)/2 \rfloor)+(\alpha + \beta + \lfloor d/2 \rfloor) + \beta = \frac{3}{2}
\alpha + 2\beta + \lfloor (\beta +d)/2 \rfloor + \lfloor d/2 \rfloor$,
labels in total, leaving us with $\alpha/2 + \beta + d - \lfloor (\beta + d)/2 \rfloor
- \lfloor d/2 \rfloor \geq \alpha/2 +\beta + d -\beta/2 -d/2 -d/2 =
\alpha/2 + \beta/2 > 0$, where in the last step we use that $\alpha +
\beta > 0$ by assumption. Therefore, after the pruning process $X$ has
at least one label.

Subcase 4.2: There are $d$ labels on $U(n)$ after the last label of
$Z$, where $d \geq \beta$, and the next regular node $X'$ in preorder after $Z$ is not a penultimate level node. In the deletion step we remove at most $p = \alpha/2 +
\beta$ labels from $Y$ and from $Z$ (by Lemma $\ref{lem:deletionstepcap}$). If
$\alpha \geq 0$ then either at least
$\alpha/2$ labels are left on each of the leaves $Y$ and $Z$; these labels are lifted up to $X$ in the lifting step. If $\alpha <0$ then at most $-2\alpha/2 = -\alpha < \beta$ labels are removed from $X$. Either way, immediately prior to the correction step that removes $\beta$ labels from the tree, $X$ is nonempty.

By assumption, at the start of the pruning process there are $\beta$ labels on $X'$. Since $X'$ is neither a leaf node nor a penultimate level node, $X'$ still has $\beta$ labels immediately before the correction step. Since $X'$ is after $X$ in preorder, the $\beta$ labels on $X'$ are in line to be removed in the correction step before any of the labels in $X$. But only $\beta$ labels are removed in the correction step, hence none of these labels come from $X$. Thus at the end of the pruning process $X$ has at least one label.

Subcase 4.3: there are $d$ labels on $U(n)$ after the last label of $Z$, where $d \geq \beta$ and the next regular node $X'$ in preorder after $Z$ is a penultimate level node. As in Subcase $4.2$, immediately before the correction step, $X$ will be nonempty. Observe that since the next regular node after the penultimate level node $X$ is another penultimate level node $X'$, the next regular node after $X'$ in preorder cannot be a third penultimate level node. Therefore, the penultimate level node $X'$ must meet the requirements of one of Case $1,2, 4.1$ or $4.2$. By the arguments of those cases, the correction step will not remove every label from $X'$. Therefore, the correction step cannot remove any labels from $X$, so after the pruning process $X$ is nonempty.

We conclude that for all $n > 4 \alpha + 5 \beta$, $M(n-\sum_{j=1}^pM(n-2j+1))$ counts the number of left leaves of
$U(n)$. Since the leaves of $U(n)$, other than possibly the last, have exactly $\gamma =
\alpha+\beta$ labels, it follows that we can replace $n$ by $n-\gamma$ in the above argument to conclude that $n$ is on or after a right leaf in $U(n)$ if and only if $n-\gamma$ is on or after the sibling left leaf. This immediately implies
that $M(n-\gamma-\sum_{j=1}^pM(n-2j+1-\gamma))$ counts the number of
right leaves of $U(n)$. Thus, $M(n-\sum_{j=1}^pM(n-2j+1)) +
M(n-\gamma-\sum_{j=1}^pM(n-2j+1-\gamma))$ counts the total number of leaves of $U(n)$, which equals $M(n)$.
\end{proof}



This section concludes with two results that relate the solutions to
low-order recursions to those of higher order. Let $b$ be an integer sequence. Define the \emph{$m$-interleaving} of $b$ to be the sequence where each term of $b$ is repeated $m$ times. For example, if $b = 1,2,2,3,\ldots$ then the $3$-interleaving of $b$ is $1,1,1,2,2,2,2,2,2,3,3,3,\ldots$. We have new notation for this notion: if $b = b_1,b_2,
b_3, \ldots $ then let $b_1^{m}$ denote $b_1$
repeated $m$ times.  The $m$-interleaving of $b$ is thus written as
$b_1^m, b_2^m, b_3^m \ldots$.
\tref{interleave} says that for any order 1 recurrence
$B$, there is an order $m$ recursion such that the
$m$-interleavings of any solution to $B$ solves $A$. These $m$-interleavings are useful, since the $m$-interleaving of the $(\alpha,\beta)$-Conolly sequence is the $(m\alpha,m\beta)$-Conolly sequence.

\begin{theorem} (Order Multiplying Interleaving Theorem)
  Let $B = \SEQ{s}{a}{t}{b}[\xi_1,\xi_2,\ldots , \xi_c]$ be a
  recurrence relation with a well defined solution sequence.  Then for any
  integer $m > 1$ the recurrence
  \begin{align*}
    A = \SEQ{ms}{(ma)^m}{mt}{(mb)^m}[\xi_1^m,\xi_2^m,\ldots , \xi_c^m]
  \end{align*}
  also has a unique solution which is the $m$-interleaving of $B$,
  where superscript $m$ denotes multiplicity.
\label{thm:interleave}
\end{theorem}
\begin{proof}
  We prove by induction on $n$ that $A(mn-j) = B(n)$ for $0\le j < m$.
  Observe that for $B$ to have a unique solution it must have at least
  one initial value $\xi_1$.  Thus for $n=1$ we have the base case, $A(m-j) =
  \xi_1 = B(1)$ and similarly for all $\xi_i$.
  Assume the theorem is true for values less than $n$. We show it holds for $n$:
\begin{eqnarray*}
  &&A\left( mn-j \right)\\
  & = & A\left( mn-ms-j-m A\left(mn-ma-j\right)\right) + A\left( mn-mt-j-m A\left(mn-mb-j\right)\right) \\
  & = & A\left( m\left(n-s-A\left(m\left(n-a\right)-j\right)\right)-j\right) + A\left( m\left(n-t-A\left(m\left(n-b\right)-j\right)\right)-j\right) \\
  & = & A\left( m\left(n-s-B\left(n-a\right)\right)-j\right) + A\left( m\left(n-t-B\left(n-b\right)\right)-j\right) \\
  & = & B\left( n-s-B\left(n-a\right) \right) + B\left( n-t-B\left(n-b\right)\right) \\
  & = & B\left( n \right)
\end{eqnarray*}
Note that for the induction to be valid we must have $s+B(n-a) > 0$
and $t+B(n-b)>0$.  This is guaranteed because we require that $B(n) >
0$ for all $n>0$, and $s, t \geq 0$.
\end{proof}



If $B$ is a slowly growing sequence and $A$ is the sequence resulting
from an application of \tref{interleave}, then certain perturbations
of the parameters of $A$ leave the solution unchanged. By taking $B$ to be a $(0,1)$-Conolly recursion, we can produce a variety of $(0,m)$-Conolly recursions by applying \tref{interleave} to $B$ and then perturbing the resulting recursion with \tref{perturb}. We could do the same with $(2,0)$-Conolly recursions, but in fact \tref{ceiling} provides a much stronger result in this case.

\begin{theorem}
  \label{thm:perturb}
  Let $B = \SEQ{s}{a}{t}{b}[\xi_1,\xi_2,\ldots,\xi_c]$
  be a slowly growing sequence.  Let $\alpha_1,\ldots,\alpha_m$ and
  $\beta_1,\ldots,\beta_m$ be integer constants that satisfy for all $1\leq i \leq m$
\begin{align} \label{eq:perturbcond}
    \begin{split}
    i-m &\le \alpha_i < i \text{ and}\\
    i-m &\le \beta_i < i.
   \end{split}
\end{align}
If the sequence
\begin{align}
 C = \SEQ{ms}{ma-\alpha_1,\ldots,ma-\alpha_m}{mt}{mb-\beta_1,\ldots,mb-\beta_m}[\xi_1^m,\xi_2^m,\ldots,\xi_c^m]
\end{align}
is well defined, then it is an $m$-interleaving of $B$.
\label{thm:perturb}
\end{theorem}
\begin{proof}
 Let $A$ be the $m$-interleaving of $B$. We prove that
  \begin{align}
    \label{eq:interalpha}
    A\left(n-ms - \sum_{i=1}^m A(n - ma + \alpha_i)\right) = A(n-ms - mA(n-ma)).
  \end{align}
  and note that a similar argument works for
    \begin{align*}
    A\left(n-mt - \sum_{i=1}^m A(n - mb + \beta_i)\right) = A(n-mt - mA(n-mb)).
  \end{align*}
  This will prove that $A = \SEQ{ms}{ma-\alpha_1,\ldots,ma-\alpha_m}{mt}{mb-\beta_1,\ldots,mb-\beta_m}[\xi_1^m,\xi_2^m,\ldots,\xi_c^m]$, by Theorem \ref{thm:interleave}.
  Note that insisting $C$ be well defined is simply a requirement that
  $c$ is large enough so that $n-ma +\alpha_i > 0$ whenever $n>mc$.

  The intuition behind this proof is that for each term referenced by $B$, the
  corresponding terms referenced by $C$ lie in some $m$-interval, that is, they belong to $[km+1,(k+1)m]$ for some positive integer $k$. This provides flexibility in the parameters of $C$.

  Let $n>mc$ be given, and set $j$ such that $1 \leq j \leq m$ and $j = n$ mod $m$. Now, observe that as $A$ is an $m$-interleaving, its values can only change at multiples of $m$, that is, if $A(n) \neq A(n+1)$, then $n = mz$ for some integer $z$. This means that in order to show (\ref{eq:interalpha}), it suffices to show that $n-ms-\sum_{i=1}^m A(n-ma + \alpha_i)$ lies in the same $m$-interval as $n-ms-m(A(n-ma))$. Observe that since $j,n$ and $n-ms-m(A(n-ma))$ are all equal mod $m$, the left and right endpoints of the $m$-interval containing $n-ms-\sum_{i=1}^m A(n-ma + \alpha_i)$ must be $n-ms-m(A(n-ma))-(j-1)$ and $n-ms-m(A(n-ma))-j+m$. Thus we infer the following inequalities:

  \begin{align}
  \label{eq:variation}
   - mA(n-ma) - (j-1) \le -\sum_{i=1}^m A(n-ma + \alpha_i) \leq -mA(n-ma) + (m-j)
  \end{align}

  To do so, we first observe that since $i-m \leq \alpha_i < i$,
  adding $\alpha_i$ can only either move back one $m$-interval, not
  change the $m$-interval, or move forward one $m$-interval. Since
  $B(n)$ is slowly growing, this means that $A_i := A(n-ma+\alpha_i) -
  A(n-ma) = 0,1$, or $-1$. We will show that $A_i$ is $1$ at most $j-1$
  times and $-1$ at most $m-j$ times, thereby establishing
  (\ref{eq:variation}).

  In order to show this, first observe that if $n$ is in the interval
  $[km+1,km+m]$, then $A_i = 1$ only if $n+\alpha_i$ is in the
  following interval $[km+m+1,km+2m]$, and $A_i = -1$ only if
  $n+\alpha_i$ is in the preceding interval $[km-m+1,km]$. By the
  definition of $j$, $n=km+j$ so $n+\alpha_i$ is in the interval
  $[km+m+1,km+2m]$ if and only if $j+ \alpha_i > m$, and $n+\alpha_i$
  is in the interval $[km-m+1,km]$ if and only if $j+\alpha_i \leq 0$.

  Thus, the number of $i$ with $A_i=1$ is at most the number of $i$
  with $j+\alpha_i>m$, and the number of $i$ with $A_i = -1$ is at
  most the number of $i$ with $j+\alpha_i \leq 0$. So it suffices to
  show that at most $j-1$ of the $\alpha_i$ satisfy $j+\alpha_i > m$,
  and at most $m-j$ of the $\alpha_i$ satisfy $j+\alpha_i \leq 0$.

  If $j+\alpha_i > m$, then $\alpha_i > m-j$, so since $\alpha_i < i$,
  this can only be true for at most the $j-1$ indices
  ${m-j+2},{m-j+3},\ldots,{m}$. Similarly, if $j+\alpha_i \leq 0$,
  then since $\alpha_i \geq i-m$, it follows that $i-m \leq \alpha_i
  \leq -j$ so $i \leq m-j$, which means $i = 1,2,\ldots,m-j$, for a
  total of only $m-j$ values. This gives us the desired bound on the
  number of indices $i$ with $A_i = 1$ or $-1$, completing the proof.

\end{proof}

\section{Enumerating Conolly-Like Recursions} \label{sec:Ceiling}

Up to this point we have focused on showing that specific recursions are $(\alpha,
\beta)$-Conolly. Here we turn to the following question: for a given $(\alpha,\beta)$, what is the complete list of $(\alpha,\beta)$-Conolly recursions of the form (\ref{eq:general2})?

Based on our experimental evidence described in Section $\ref{sec:Exp}$, we believe that for each possible
$(\alpha,\beta)$ pair with $\beta>0$ there are finitely many (and at least 1) $2$-ary $(\alpha,\beta)$-Conolly recursions. Below we illustrate this hypothesis in the case $p=2$ (see Conjecture \ref{conj:order2}), where we list what we believe are all the $2$-ary recursions. Note that the $(\alpha,\beta)$ pairs we list in Conjecture \ref{conj:order2} appear in Table \ref{tbl:Pairs}; in the tables in Conjecture \ref{conj:order2} the set notation denotes that parameters can be chosen from the Cartesian product of the sets, while the right hand column in the tables is simply the size of that product.

\begin{conjecture}
\label{conj:order2}
 For $\beta > 0$, the only $2$-ary, order $2$ $(\alpha,\beta)$-Conolly recurrences are:

For ($\alpha, \beta)=(-2,3)$: \\
\begin{center}
\begin{tabular}{l|r}
Recurrence & number \\ \hline
$\SEQ{0}{1,3}{1}{2,4}$ & 1 \\
$\SEQ{0}{2,3}{3}{4,\{7,8,9\}}$ & 3 \\
$\SEQ{0}{\{2,3,4\},\{4,5,6\}}{3}{\{2,3\},9}$ & 18 \\
$\SEQ{0}{\{2,3,4\},\{4,5,6\}}{5}{\{7,8,9\},\{9,10,11\}}$ & 81 \\
\end{tabular}
\end{center}

For ($\alpha, \beta)=(0,2)$: \\
\begin{center}
\begin{tabular}{l|r}
Recurrence & number \\ \hline
$\SEQ{0}{\{1,2\},\{2,3\}}{1}{1,4}$ & 4 \\
$\SEQ{0}{\{1,2\},\{2,3\}}{2}{\{3,4\},\{4,5\}}$ & 16 \\
$\SEQ{0}{\{3,4\},\{4,5\}}{4}{3,10}$ & 4 \\
$\SEQ{0}{\{3,4\},\{4,5\}}{6}{\{9,10\},\{10,11\}}$ & 16 \\
\end{tabular}
\end{center}

For ($\alpha, \beta)=(2,1)$: \\
\begin{center}
\begin{tabular}{l|r}
Recurrence & number \\ \hline
$\SEQ{0}{1,1}{1}{2,2}$ & 1 \\
$\SEQ{0}{\{1,2\},\{2,3\}}{1}{1,2}$ & 4 \\
$\SEQ{0}{\{1,2\},\{2,3\}}{2}{1,\{5,6\}}$ & 8 \\
$\SEQ{0}{\{1,2\},\{2,3\}}{2}{2,\{4,5\}}$ & 8 \\
$\SEQ{0}{\{1,2\},\{2,3\}}{3}{\{4,5\},\{5,6\}}$ & 16 \\
\end{tabular}
\end{center}

\end{conjecture}

When $\beta = 0$, the situation is very different. We have the following result that holds for recursions of the form ($\ref{eq:generalRecurrence}$) with arbitrary arity $k$ and any values of the parameters $p_i$:

\begin{theorem}
\label{thm:alphazeroconollylike}
For all $\alpha > 0$, there are either no  $(\alpha, 0)$-Conolly meta-Fibonacci recursions of the form ($\ref{eq:generalRecurrence}$), or there are infinitely many.
\end{theorem}
\begin{proof}
Suppose that for some given $\alpha$ and set of parameters the recursion
\begin{align}
R(n) = \sum_{i=1}^k R\left(n-s_i-\sum_{j=1}^{p_i} R(n-a_{ij})\right)
\label{eq:alphazeroconollylike}
\end{align}
together with $c$ initial values has an $(\alpha,0)$-Conolly solution. Since the $(\alpha, 0)$-Conolly solution sequence assumes the value of each positive integer $\alpha$ times in order, it must be the sequence $\lceil n/\alpha \rceil$. Thus, $R(n) = \lceil n/\alpha \rceil$ satisfies ($\ref{eq:alphazeroconollylike}$), the $c$ initial values must be the first $c$ values of the sequence $\lceil n/\alpha \rceil$, and for $n>c$ we know that the arguments $n-a_{ij}>0$ and $n > n-s_i-\sum_{j=1}^{p_i} \lceil \frac{n-a_{ij}}{\alpha}\rceil > 0$.

Now we define a new, related recurrence with an $(\alpha,0)$-Conolly
solution as follows:  Set
\begin{align}
P(n) = \sum_{i=1}^k P\left(n-s_i-p_i-\sum_{j=1}^{p_i} P(n-a_{ij}-\alpha)\right),
\label{eq:shiftedalphazero}
\end{align}
with $c+\alpha$ initial values which
we take to be the first $c+\alpha$ values of the sequence
$\lceil n/\alpha \rceil$. Observe that from the discussion in the paragraph above $P(n-a_{ij}-\alpha)$ is well-defined. We show below that in fact for $n>c+\alpha$ all the arguments of $P$ on the right hand side of (\ref{eq:shiftedalphazero}) are positive so the terms are well-defined.

We now show by induction that for all $n$,
$P(n) = \lceil n/\alpha \rceil$. By our assumption for the initial conditions we have the base case.
Assume that our hypothesis is true up to $n-1$. Then

\begin{align*}
  &\sum_{i=1}^k P\left(n-s_i-p_i-\sum_{j=1}^{p_i}
    P(n-a_{ij}-\alpha)\right) = \sum_{i=1}^k
  P\left(n-s_i-p_i-\sum_{j=1}^{p_i} (P(n-a_{ij})-1)\right)\\
  =& \sum_{i=1}^k P\left(n-s_i-p_i+p_i-\sum_{j=1}^{p_i}
    P(n-a_{ij})\right) = \sum_{i=1}^k P\left(n-s_i-\sum_{j=1}^{p_i}
    P(n-a_{ij})\right) = \cln{n}{\alpha}.
\end{align*}
The last equality holds by our assumption
that ($\ref{eq:alphazeroconollylike}$) has solution sequence $\lceil
n/\alpha \rceil$ and by the fact that $n>n-s_i-\sum_{j=1}^{p_i}
P(n-a_{ij})>0$ since $P(n-a_{ij}) = \lceil
\frac{n-a_{ij}}{\alpha}\rceil$. This shows that $\lceil n/\alpha
\rceil$ solves ($\ref{eq:shiftedalphazero}$).

This proves that $P(n)$ has an $(\alpha,0)$-Conolly solution. We could repeat this process, replacing $R(n)$ with $P(n)$, and hence construct infinitely many different recursions with an $(\alpha, 0)$-Conolly solution, thereby proving the desired result.

\end{proof}

For example, in \cite{BLT} it is shown that $\SEQ{0}{1}{2}{3}$ is $(2,0)$-Conolly; by the above theorem we deduce that so too are $\SEQ{1}{3}{3}{5}$, $\SEQ{2}{5}{4}{7}$ and in general, $\SEQ{x}{2x+1}{x+2}{2x+3}$ for any $x \geq 0$.

We are now better equipped to return to the case ($\alpha, \beta)=(4,0)$, the only ($\alpha, \beta)$ pair not yet covered in our discussion of $2$-ary recursions with $p=2$. Applying Theorems \ref{thm:alphabeta} and $\ref{thm:alphazeroconollylike}$ we conclude that there are an infinite number of $(4,0)$-Conolly recursions.

It is evident that we can use the same reasoning as above for $p=2$ to show that there are an infinite number of $(2p,0)$-Conolly $2$-ary recursions for any $p$. In Theorem $\ref{thm:ceiling}$ following, we substantially improve this result by providing a complete list of all such recursions.

Central to Theorem $\ref{thm:ceiling}$ is the observation that since the $(\alpha,0)$-Conolly sequences have (by definition) constant frequency sequences with value
$\alpha$, they are equal to $\cln{n}{\alpha}$.  By \cref{abConditions}, if $H$ is of the form (\ref{eq:general2}) with an $(\alpha,0)$-Conolly
solution, then $\alpha=2p$. In what follows we provide necessary and
sufficient conditions for the sequence $\cln{n}{2p}$ to be the
solution of $H$.

For technical reasons, we need to distinguish between the property that a meta-Fibonacci recursion $A(n)$ with given initial conditions generates $B(n)$ as its (unique) solution sequence via a recursive calculation, and the property that the sequence $B(n)$ \emph{formally satisfies} the recursion $A(n)$, by which we mean that for all $n$, $B(n)$ satisfies the equation that defines $A(n)$. An example makes this distinction clearer: consider the recursion $R=\SEQ{-1}{-1}{2}{3}$. $R$ is formally satisfied by the sequence $\cln{n}{2}$, but $\cln{n}{2}$ is not the solution sequence to $R$ (for any set of initial conditions) because the recursion $R(n) = R(n+1-R(n+1))-R(n-2-R(n-3))$ will always require that we know the term $R(n+1)$ to calculate $R(n)$.

Note that in the example above the recursion $R$ has some negative parameters, a situation that we don't normally permit. As we will see in Corollary \ref{cor:produceceiling} no such example is possible without negative parameters.

Our strategy for classifying $(\alpha, 0)$-Conolly sequences takes advantage of this distinction. We first show that the $(\alpha,0)$-Conolly sequence formally satisfies a $2$-ary recursion if and only if that recursion's parameters meet a certain set of conditions described in Theorem \ref{thm:ceiling} below. Then we show that formal satisfaction is equivalent to generating the sequence as the solution to the recurrence so long as we provide sufficiently many initial conditions that match the ceiling function.

For any integer $z$, let $\quo{z}$ and $\rem{z}$ be the quotient and remainder mod $2p$, so that $z = 2p\quo{z} + \rem{z}$ with $0\le \rem{z} < 2p$.

\begin{theorem}
  \label{thm:ceiling}
  The $2$-ary, order $p$ meta-Fibonacci recurrence relation
  \begin{align*}
    H(n) = H\left(n-s - \sum_{i=1}^p H(n-a_i) \right) + H\left(n-t -
      \sum_{i=1}^p H(n-b_i) \right)
  \end{align*}
  is formally satisfied by the sequence $\cln{n}{2p}$ for all integers $n$ if
  and only if the parameters satisfy
  \begin{enumerate}
  \item For each integer $j$ in $\{0,1,\ldots,p-1\}$, at most $j$ of
    the $\rem{a}_i$s satisfy $\rem{a}_i \le j$ and at most $j$ of them satisfy
    $\rem{a}_i \ge 2p - j$.
  \item For each integer $j$ in $\{0,1,\ldots,p-1\}$, at most $j$ of
    the $\rem{b}_i$s satisfy $\rem{b}_i \le j$ and at most $j$ of them satisfy
    $\rem{b}_i \ge 2p - j$.
  \item There exists an integer $d$ such that either $-s + \sum_{i=1}^p
    \quo{a}_i = 2pd$ and $-t + \sum_{i=1}^p \quo{b}_i = -2pd -p$, or $-s +
    \sum_{i=1}^p \quo{a}_i = -2pd-p$ and $-t + \sum_{i=1}^p \quo{b}_i = 2pd$.
  \end{enumerate}
\end{theorem}

Note that conditions (1), (2) and (3) are invariant under the transformation from $R$ to $P$ presented in Theorem $\ref{thm:alphazeroconollylike}$ above, as well as under related transformations.

\proof

Clearly $H(n)$ is formally satisfied by $\cln{n}{2p}$ if and only if
\begin{align*}
  \cln{n}{2p} =  \cln{ n - s - \sum_{i=1}^p\cln{n-a_i}{2p} }{2p}+ \cln{ n -
    t - \sum_{i=1}^p\cln{n-b_i}{2p} }{2p}.
\end{align*}

Let $h(n) = h_1(n) + h_2(n)$ where
\begin{align*}
  h_1(n) &=  \cln{ n - s - \sum_{i=1}^p\cln{n-a_i}{2p} }{2p},  \\
  \text{ and }h_2(n) &= \cln{ n -
    t - \sum_{i=1}^p\cln{n-b_i}{2p} }{2p}.
\end{align*}

Our strategy is to show that if $h(n) = \cln{n}{2p}$, then, up to interchanging $h_1$ and $h_2$, $h_1(n) = \cln{n}{4p} + d$ and $h_2(n) = \cln{n-2p}{4p} - d$ which forces the stated conditions on the
parameters.  This is done in several steps, outlined as follows.

\lref{hislow} shows that the successive differences of both $h_1(n)$
and $h_2(n)$ are always either $-1, 0$ or $1$, regardless of the
values of the parameters $s,t,a_i$ and $b_i$.  We assume that $h(n) =
\cln{n}{2p}$, which is slowly growing, and in particular, $h(n+1) -
h(n) = 1$ if and only if $n = 2p\mu$ for some integer $\mu$.  This,
together with \lref{hislow} implies that for a given integer $\mu$,
exactly one of $h_1$ or $h_2$ satisfies $h_i(2p\mu+1)-h_i(2p\mu) = 1$.

\lref{h1conditions} and \lref{h1const} impose conditions on $h_1(n)$
for certain intervals, while Lemmas \ref{lem:h2conditions} and
\ref{lem:h2const} do the same for $h_2(n)$ in complementary intervals
(see \fref{ceiling}). We prove that $h_1(n) = \cln{n}{4p} + d$ in the
intervals where the remainder of $n$ modulo $4p$ lies in the range
$(p,3p]$, which for convenience we write as $p < n \pmod{4p}\le
3p$. We will show that in these intervals $h_1(n)$ is constant but
$h(n)$ is not.  Similarly, Lemmas \ref{lem:h2conditions} and
\ref{lem:h2const} prove that $h_2(n) = \cln{n-2p}{4p} - d$ for the
complementary intervals, $-p < n \pmod{4p}\le p $, in which $h_2(n)$
is constant but $h(n)$ is not.

\begin{figure}
\begin{center}
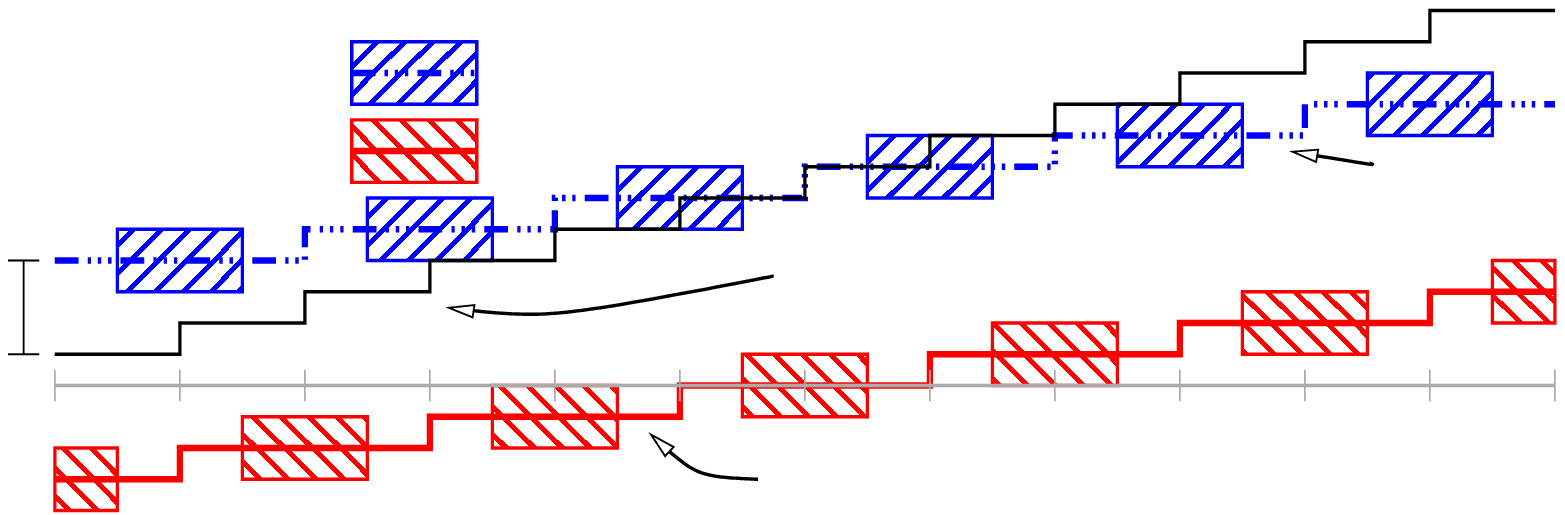
\caption{A visual representation of the proof of
  \tref{ceiling}. Lemmas \ref{lem:h1conditions} and \ref{lem:h1const}
  prove that $h_1(n)$ is constant in the intervals
  highlighted on the curve labeled $h_1(n)$.  Lemmas
  \ref{lem:h2conditions} and \ref{lem:h2const} operate similarly on
  intervals complementary to these, showing that $h_2(n)$ is a
  constant.  We assume that $h(n) = \cln{n}{2p}$, and use
  \lref{hislow}, to show that $h_1(n) = \cln{n}{4p} + d$ and
  $h_2(n)=\cln{n-2p}{4p}-d$, for all $n$}
\label{fig:ceiling}
\end{center}
\end{figure}

Thus for each interval, we know the values of $h(n)$ and exactly one
of $h_1(n)$ or $h_2(n)$.  Using $h(n) = h_1(n) + h_2(n)$ we are able
to determine that $h_1(n) = \cln{n}{4p} + d$ and $h_2(n) =
\cln{n-2p}{4p} - d$, for all $n$.  This imposes the stated conditions
on the parameters $s,t,a_i$ and $b_i$.  The converse will also become
evident.  That is, if we assume the parameters $s,t$, $a_i$ and $b_i$
satisfy conditions (1), (2) and (3), then $H(n)$ is satisfied by
$\cln{n}{2p}$.

\begin{lemma}\label{lem:hislow}
For all $n$, $|h_i(n + 1) - h_i(n)| \leq 1$.
\end{lemma}
\begin{proof}
  We show that the numerator in $h_1$ does not change too
  quickly, which ensures that $h_1$ has successive
  differences that are small in magnitude.

  Observe that $-1\le\cln{n-a_i}{2p} - \cln{n+1-a_i}{2p} \le 0$, so
  that
  \begin{align*}
    -p\le \sum_{i=1}^p \left( \cln{n-a_i}{2p} - \cln{n+1-a_i}{2p}\right) \le 0.
  \end{align*}
  Subtracting the numerators for successive arguments of $h_1$ we get
  \begin{align*}
    \left\lvert n+1 - s - \sum_{i=1}^p\cln{n+1-a_i}{2p} - n + s +
      \sum_{i=1}^p\cln{n-a_i}{2p} \right\rvert\\
    = \left\lvert 1 + \sum_{i=1}^p\left(\cln{n-a_i}{2p} -
        \cln{n+1-a_i}{2p}\right)\right\rvert < p < 2p.
  \end{align*}
  Thus $|h_1(n+1) - h_1(n)| \le 1$.
  The proof for $h_2$ is similar, just replace $s$ with $t$ and
  $a_i$ with $b_i$.
  \end{proof}

 If $h(n+1) - h(n) = 1$, then necessarily $h_1(n+1)-h_1(n)>0$ or
 $h_2(n+1)-h_2(n)>0$.  By \lref{hislow} either $h_1(n+1)-h_1(n)=1$ and $h_2(n+1)-h_2(n)=0$ or vice versa. Without loss of generality, since we can interchange $h_1$ and $h_2$, we assume that
 \begin{align}\label{eq:assume}
   h_1(1) - h_1(0) = 1 \text{ and } h_2(1) - h_2(0) = 0.
 \end{align}

 It is helpful to expand the numerator of $h_1(n)$ using the notation $\quo{z}$ and $\rem{z}$ for the quotient and remainder modulo $2p$ introduced above:
\begin{align}
\label{eq:prelimh1:a}  h_1(n) &= \cln{ 2p\quo{n} + \rem{n} - s - \sum_{i=1}^p\cln{2p\quo{n} + \rem{n} - 2p\quo{a}_i - \rem{a}_i}{2p} }{2p} \\
\notag &= \cln{ 2p\quo{n} + \rem{n} - s - \sum_{i=1}^p\left(\quo{n}-\quo{a}_i + \cln{\rem{n} - \rem{a}_i}{2p} \right)}{2p} \\
\label{eq:prelimh1}  &= \cln{ p\quo{n} + \rem{n} - s + \sum_{i=1}^p \quo{a}_i- \sum_{i=1}^p\cln{\rem{n}-\rem{a}_i}{2p} }{2p}.
\end{align}
A similar expression holds for $h_2$, where we replace $s$ with $t$ and $a_i$ with $b_i$.

Observe that the sum $\sum_{i=1}^p\cln{\rem{n}-\rem{a}_i}{2p}$ in (\ref{eq:prelimh1}) is
quite tame.  In particular, $0\le \rem{n} < 2p$ and $0 \le \rem{a}_i < 2p$ so
that $\cln{\rem{n}-\rem{a}_i}{2p}$ is either 0 or 1, and is equal to $1$ if and only if
$\rem{n}$ is strictly greater than $\rem{a}_i$.  Using the standard Iversonian notation, we write
\begin{align*}
  \cln{\rem{n}-\rem{a}_i}{2p} = \I{\rem{n}>\rem{a}_i}.
\end{align*}
Using this notation we have the inequality $0\le \sum_{i=1}^p \I{
  \rem{n} > \rem{a}_i } \le p$, which is used extensively in the
lemmas below to describe the behavior of $h_i$ for certain intervals.
We can combine the Iversonian notation with ($\ref{eq:prelimh1}$) to
rewrite:
\begin{align}
\label{eq:numh1}  h_1(n) = \cln{ p\quo{n} + \rem{n} - s + \sum_{i=1}^p \quo{a}_i- \sum_{i=1}^p\I{\rem{n}>\rem{a}_i} }{2p}.
\end{align}

Let $d = h_1(0)$, so that $h_2(0) = h(0) - h_1(0) = -d$.  By (\ref{eq:assume}) $h_1(1) = d+1$ while $h_2(1) = -d$.

\begin{lemma}\label{lem:h1conditions}
  For each $i$ in $\{1, 2 , \ldots, p\}$, $\rem{a}_i \neq 0$. Further, $-s + \sum_{i=1}^p \quo{a}_i = 2pd$.
\end{lemma}
\begin{proof}
  We compare numerators in $h_1(0)$ and $h_1(1)$ using the form
  derived in (\ref{eq:numh1}).  Notice that $\I{0>\rem{a}_i} = 0$ for
  every $i$, so that
\begin{align*}
h_1(0) = \cln{-s + \sum_{i=1}^p \quo{a}_i}{2p}.
\end{align*}

Since $h_1(1) - h_1(0) = 1$ and
\begin{align*}
  h_1(1) = \cln{1-s + \sum_{i=1}^p\quo{a}_i -
    \sum_{i=1}^p\I{1>\rem{a}_i}}{2p},
\end{align*}
the numerator of $h_1(1)$ must be greater than the numerator of $h_1(0)$, that is,
\begin{align*}
  -s + \sum_{i=1}^p\quo{a}_i < 1 -s + \sum_{i=1}^p\quo{a}_i
  -\sum_{i=1}^p\I{1>\rem{a}_i},
\end{align*}
which simplifies to
\begin{align*}
  1 - \sum_{i=1}^p\I{1>\rem{a}_i} > 0.
\end{align*}
It follows that for each $1\le i \le p$ we have $\I{1>\rem{a}_i} = 0$ so $\rem{a}_i \ge 1$.
Thus, $h_1(1) = \cln{1-s +
  \sum_{i=1}^p\quo{a}_i}{2p} = d+1$ and $h_1(0) = \cln{-s +
  \sum_{i=1}^p\quo{a}_i}{2p} = d$ which implies that $-s +
\sum_{i=1}^p\quo{a}_i = 2pd$, as required.
\end{proof}

We can now characterize the behavior of $h_1(n)$ whenever $p < n \pmod{4p}\le 3p$.

\begin{lemma}
\label{lem:h1const}
With $d$ as defined above (right before the statement of \lref{h1conditions}),
$$h_1(n) = \cln{n}{4p} + d$$

whenever $p < n \pmod{4p}\le 3p $.
\end{lemma}
\begin{proof}

  We rewrite (\ref{eq:numh1}) using $ -s + \sum_{i=1}^p\quo{a}_i = 2pd$ from Lemma \ref{lem:h1conditions}:
  \begin{align*}
    h_1(n)& = \cln{ p\quo{n} + \rem{n} + 2pd - \sum_{i=1}^p\I{\rem{n}>\rem{a}_i} }{2p}\\
    &= d + \cln{ p\quo{n} + \rem{n} - \sum_{i=1}^p\I{\rem{n}>\rem{a}_i}
    }{2p}.
  \end{align*}
  To complete the proof of the lemma, we check that the ceiling function above is in fact $\cln{n}{4p}$ as desired. We separate into three cases:

  Case 1: $p < n \pmod{4p} < 2p$. Observe that in this case, $\quo{n}$ is even, so
  \begin{align*}
    \cln{ p\quo{n} + \rem{n} - \sum_{i=1}^p\I{\rem{n}>\rem{a}_i}} {2p}
    =& \frac{\quo{n}}{2} + \cln{\rem{n} -
      \sum_{i=1}^p\I{\rem{n}>\rem{a}_i}}{2p}\\
    =& \frac{n-\rem{n}}{4p} +
    \cln{\rem{n} - \sum_{i=1}^p\I{\rem{n}>\rem{a}_i}}{2p}.
  \end{align*}
  Since $2p > \rem{n} > p$, we have that $\frac{n-\rem{n}}{4p} =
  \cln{n}{4p}-1$, so we simply need to note that $\cln{\rem{n} -
    \sum_{i=1}^p\I{\rem{n}>\rem{a}_i}}{2p} = 1$ because $2p > \rem{n}
  > p$ and $\sum_{i=1}^p\I{\rem{n}>\rem{a}_i} \le p$.

  Case 2: $n \pmod{4p} = 2p$. This means $\rem{n} = 0$, so
  \begin{align*}
    \cln{ p\quo{n} + \rem{n} - \sum_{i=1}^p\I{\rem{n}>\rem{a}_i} }{2p}
    = \cln{ p\quo{n} }{2p} = \cln{2p\quo{n}}{4p}= \cln{n}{4p}.
  \end{align*}

  Case 3: $2p < n \pmod{4p} \leq 3p$. Observe that in this case $\quo{n}$ is odd. Hence

  \begin{align*}
    \cln{ p\quo{n} + \rem{n} - \sum_{i=1}^p\I{\rem{n}>\rem{a}_i}}{2p}
    =& \cln{ p\quo{n}+p-p + \rem{n} -
      \sum_{i=1}^p\I{\rem{n}>\rem{a}_i}}{2p}\\
    =& \frac{\quo{n}+1}{2} +
    \cln{-p+\rem{n} - \sum_{i=1}^p\I{\rem{n}>\rem{a}_i}}{2p}.
  \end{align*}
  Since $0<\rem{n} \leq p$, we have $-2p < -p+\rem{n} -
  \sum_{i=1}^p\I{\rem{n}>\rem{a}_i} \le 0$, so that
  \begin{align*}
    \cln{-p+\rem{n} - \sum_{i=1}^p\I{\rem{n}>\rem{a}_i}}{2p}=0.
  \end{align*}
 Lastly, since $p \geq
  \rem{n} > 0$, it is immediate that $\frac{\quo{n}+1}{2} =
  \cln{n}{4p}$.
\end{proof}

By Lemma \ref{lem:h1const} we know that $h_1(2p+1) - h_1(2p) = 0$.
Since $h(2p+1) - h(2p) = 1$, we have that
$h_2(2p+1) - h_2(2p) =1$.  Just as we used our assumption that $h_1(1)-h_1(0)=1$ in the above two lemmas, we apply this crucial fact to prove the corresponding two lemmas for $h_2$ that follow. The technical details of the proofs are similar to those of the two preceding lemmas.

\begin{lemma}\label{lem:h2conditions}
  For each $i$ in $\{1, 2 , \ldots, p\}$,
$\rem{b}_i \neq 0$.
Further, with $d$ as defined above,
$-t + \sum_{i=1}^p \quo{b}_i = -2pd - p$.
\end{lemma}
\begin{proof}
  Recall that if $n=2p$, then $\quo{n} = 1$ and $\rem{n} = 0$.
  Following the proof of \lref{h1conditions} we compare the numerators
  of $h_2(2p)$ and $h_2(2p+1)$.  Using (\ref{eq:numh1}) and
  $\I{0>\rem{b}_i} = 0$, we have
\begin{align*}
  h_2(2p) = \cln{p-t + \sum_{i=1}^p\quo{b}_i}{2p}.
\end{align*}
Since $h_2(2p+1) - h_2(2p) = 1$ and
\begin{align*}
  h_2(2p+1) = \cln{p + 1-t + \sum_{i=1}^p\quo{b}_i -
    \sum_{i=1}^p\I{1>\rem{b}_i}}{2p},
\end{align*}
the numerator of $h_2(2p+1)$ must be greater than that of $h_2(2p)$.
This inequality,
\begin{align*}
  p-t + \sum_{i=1}^p\quo{b}_i < p+1 -t + \sum_{i=1}^p\quo{b}_i
  -\sum_{i=1}^p\I{1>\rem{b}_i}
\end{align*}
simplifies to
\begin{align*}
1 - \sum_{i=1}^p\I{1>\rem{b}_i} > 0,
\end{align*}
implying that $\I{1>\rem{b}_i} = 0$, so $\rem{b}_i \ge 1$ for each
$1\le i \le p$. for each $1\le i \le p$.  Thus, $h_2(2p+1) =
\cln{p+1-t + \sum_{i=1}^p\quo{b}_i}{2p} = -d+1$ and $h_2(2p) =
\cln{p-t + \sum_{i=1}^p\quo{b}_i}{2p} = -d$.  Applying \lref{h1const}
and $h(n) = h_1(n) + h_2(n)$, we have $p-t + \sum_{i=1}^p\quo{b}_i =
-2pd$.
\end{proof}

Note that Lemmas \ref{lem:h2conditions} and \ref{lem:h1conditions}
establish condition (3) of \tref{ceiling}.

We now prove the analogue to \lref{h1const} for $h_2$ that we promised above.
\begin{lemma}
\label{lem:h2const}
Whenever $-p < n \pmod{4p}\le p$, then $h_2(n) = \cln{n-2p}{4p} - d$.
\end{lemma}
\proof Let $g(n) = h_2(n+2p)$. We prove that $g(n) = \cln{n}{4p} - d$
whenever $p < n \pmod{4p}\le 3p$. Using the analogue of (\ref{eq:prelimh1:a}) for $h_2$ and simplifying we get
\begin{align*}
  h_2(n+2p) = \cln{ 2p(\quo{n}+1) + \rem{n} - t - \sum_{i=1}^p\cln{2p(\quo{n}+1) + \rem{n} - 2p\quo{b}_i - \rem{b}_i}{2p} }{2p} \\
  h_2(n+2p)= \cln{ p\quo{n} + p + \rem{n} - t + \sum_{i=1}^p \quo{b}_i- \sum_{i=1}^p\I{\rem{n}>\rem{b}_i} }{2p}.\\
\end{align*}
Recall from Lemma \ref{lem:h2conditions} that $p -t +
\sum_{i=1}^p\quo{b}_i = -2pd$.  Thus, the remainder of the proof
follows closely that for $h_1(n)$ in Lemma \ref{lem:h1const}.
Therefore $g(n) = h_2(n+2p) = \cln{n}{4p} - d$ so that $h_2(n) =
\cln{n-2p}{4p}-d$ for $-p< n \pmod{4p}\le p$.\qed

Using the fact that $h_1$ and $h_2$ are constant for complementary
ranges of $n$ and that they are related by $h(n)=h_1(n) + h_2(n)$ we are able to
characterize their behavior for all $n$.

\begin{lemma} For all $n$, $  h_1(n) = \cln{n}{4p} +d$ and $h_2(n) = \cln{n-2p}{4p} -d.$
\end{lemma}
\begin{proof}
  It is easy to see that for all $n$ the ceiling function satisfies
  \begin{align*}
    h(n) = \cln{n}{2p} = \left(\cln{n}{4p}+ d\right) + \left(\cln{n-2p}{4p} - d\right).
  \end{align*}
  When $p< n \pmod{4p}\le 3p $ we have $h_1(n) = \cln{n}{4p} + d$ by
  \lref{h1const}; this implies that $h_2(n) = \cln{n-2p}{4p}-d$ in
  this range.  When $-p< n \pmod{4p}\le p $ we have $h_2(n) =
  \cln{n-2p}{4p}-d$ by \lref{h2const}; this implies that $h_1(n) =
  \cln{n}{4p} + d$ in this range.
\end{proof}



Now that we have proven that for all $n$, $h_1(n)=\cln{n}{4p}+d$, we can show that condition (1) on the $\rem{a}_i$ is necessary. Recall that in the proof of Lemma \ref{lem:h1const}, we showed for all $n$, $h_1(n) = d + \cln{p\quo{n} + \rem{n} - \sum_{i=1}^p\I{\rem{n}>\rem{a}_i}}{2p}$. Therefore we have that for all $n$,
\begin{align}
\label{eq:ceilingequality}
\cln{n}{4p}=\cln{p\quo{n} + \rem{n} - \sum_{i=1}^p\I{\rem{n}>\rem{a}_i}}{2p}.
\end{align}
Our proof is driven by this ceiling equality. In fact, we only need to use $n$ in the range $0 < n < 4p$ to get the desired results; the two segments $0 < n < 2p$ and $2p < n < 4p$ will force the two different parts of condition (1).

First, consider $n$ in the range $0 < n < 2p$. In this case $\rem{n} =
n$ and $\quo{n} = 0$. So (\ref{eq:ceilingequality}) implies
$\cln{n-\sum_{i=1}^p\I{\rem{n}>\rem{a}_i}}{2p} = 1$, so $0 <
n-\sum_{i=1}^p\I{\rem{n}>\rem{a}_i} \leq 2p$. From the left inequality
we get $\sum_{i=1}^p\I{n>\rem{a}_i} < n$.  It is easy to see how this
implies the first part of condition $(1)$: for each integer $j$ in
$\{0,1,\ldots,p-1\}$, let $j+1=n$. Then by the preceding inequality
$\sum_{i=1}^p\I{j\ge \rem{a}_i} \le j$.  That is, for each integer $j$
in $\{0,1,\ldots,p-1\}$, at most $j$ of the $\rem{a}_i$s satisfy
$\rem{a}_i \le j$.  Therefore this condition is necessary for \eref{ceilingequality} to be satisfied for $n$ in the
range $0<n\pmod{4p}<2p$.

Next, consider $2p < n < 4p$. Here, $\rem{n} = n-2p$ and $\quo{n} =
1$. Substituting these into (\ref{eq:ceilingequality}) we get
$\cln{p+n-2p-\sum_{i=1}^p\I{\rem{n}>\rem{a}_i}}{2p} = 1$, which is
equivalent to $0 < n-p-\sum_{i=1}^p\I{\rem{n}>\rem{a}_i} \leq 2p$. The
right inequality implies $n -3p \leq \sum_{i=1}^p \I{n-2p >
  \rem{a}_i}$ so when $j = 4p-n$, we have $p-j \leq \sum_{i=1}^p
\I{2p-j > \rem{a}_i}$.  That is, for each integer $j$ in
$\{1,2,\ldots,p-1\}$, at most $j$ of the $\rem{a}_i$s satisfy
$\rem{a}_i \ge 2p-j$. Therefore, this condition is necessary for
\eref{ceilingequality} to be satisfied for $n$ in the range
$2p<n\pmod{4p}<4p$.  In addition, $\rem{a}_i$ is a remainder modulo
$2p$, so the condition also holds for $j=0$, proving the necessity of
the second part of condition $(1)$ of the theorem.

If we substitute $m = n-2p$ into $h_2(n) = \cln{n-2p}{4p} -d$ we can
utilize the above argument to prove condition $(2)$ of \tref{ceiling}.

We have established the necessity of (1), (2) and (3). We now show sufficiency. Our strategy is simple: we will reverse our arguments to show that conditions (1),(2), and (3) imply that $h_1(n)$ and $h_2(n)$ are $\cln{n}{4p}+d$ and $\cln{n-2p}{4p} - d$. Assume that all three conditions hold. Without loss of generality, assume that $-s + \sum_{i=1}^{p}{\quo{a_i}} = 2pd$. If not, switch the first and second summands in the recursion.

As shown previously, expanding the definition of $h_1(n)$ gives \eref{numh1}, which for convenience we rewrite below:
\[
h_1(n) = \cln{ p\quo{n} + \rem{n} - s + \sum_{i=1}^p \quo{a}_i- \sum_{i=1}^p\I{\rem{n}>\rem{a}_i} }{2p}.
\]

Substituting $2pd$ for $-s+\sum_{i=1}^{p}{\quo{a_i}}$, then factoring out $d$ and bringing $d$ out of the ceiling function we get $h_1(n) = d+\cln{ p\quo{n} + \rem{n} - \sum_{i=1}^p\I{\rem{n}>\rem{a}_i} }{2p}$. A similar equation holds for $h_2(n)$. Thus, if we can show that $\cln{ p\quo{n} + \rem{n} - \sum_{i=1}^p\I{\rem{n}>\rem{a}_i} }{2p} = \cln{n}{4p}$, with a corresponding equality for $h_2$, we will be done. We show the details only for the required equality above for $h_1$; the approach in the second case is entirely similar.

First, note that if the desired equality above for $h_1$ holds for $n$, it also holds for $n+4p$. This is because $\cln{ p\quo{(n+4p)} + \rem{n} - \sum_{i=1}^p\I{\rem{n}>\rem{a}_i} }{2p} = \cln{ p\quo{n}+2p + \rem{n} - \sum_{i=1}^p\I{\rem{n}>\rem{a}_i} }{2p} = 1+\cln{ p\quo{n} + \rem{n} - \sum_{i=1}^p\I{\rem{n}>\rem{a}_i} }{2p} = 1+\cln{n}{4p} = \cln{n+4p}{4p}$. Thus, without loss of generality we consider only $n$ with $1 \leq n \leq 4p$.

For all such $n$, $\cln{n}{4p} = 1$, so we need only prove that for $1 \leq n \leq 4p$, $\cln{ p\quo{n} + \rem{n} - \sum_{i=1}^p\I{\rem{n}>\rem{a}_i} }{2p} = 1$. This is equivalent to proving that for $1 \leq n \leq 4p$, $1 \leq p\quo{n} + \rem{n} - \sum_{i=1}^p\I{\rem{n}>\rem{a}_i} \leq 2p$. We consider three cases.

Case 1: $1 \leq n < 2p$. In this case $\quo{n} = 0, \rem{n} = n$. We want $1 \leq n - \sum_{i=1}^p\I{n>\rem{a}_i} \leq 2p$. Since $\I{n>\rem{a}_i} = 1-\I{n\leq\rem{a}_i}$, the prior inequality is equivalent to $1 \leq n - \sum_{i=1}^p(1-\I{n\leq\rem{a}_i}) \leq 2p$, or $1 \leq n-p+\sum_{i=1}^p\I{n\leq\rem{a}_i} \leq 2p$, which we can rearrange as $1+p-n \leq \sum_{i=1}^p\I{n\leq\rem{a}_i} \leq 3p-n$. Because $n < 2p$, the right-hand inequality is always true, and clearly the left-hand inequality can only possibly be false for $n \leq p$. By condition (1), at most $n-1$ of the $\rem{a_i}$ satisfy $\rem{a_i} \leq n-1$, which is equivalent to $\rem{a_i} < n$. Therefore, at least $p-(n-1)$ of the $\rem{a_i}$ satisfy the negation, $n \leq \rem{a_i}$, so $1+p-n \leq \sum_{i=1}^p(\I{n\leq\rem{a}_i})$.

Case 2: $n = 2p, 4p$. In this case, $\quo{n} = 1$ or $2$, and $\rem{n} = 0$, which makes $p\quo{n} + \rem{n} - \sum_{i=1}^p\I{\rem{n}>\rem{a}_i} = p$ or $2p$, satisfying the required inequality.

Case 3: $2p < n < 4p$. In this case, $\quo{n} = 1$ and $\rem{n} = n-2p$. Thus, the required inequality is $1 \leq p + n-2p - \sum_{i=1}^p\I{n-2p>\rem{a}_i} \leq 2p$. The left inequality is always true since $n > 2p$. So we need only show that  $n-p-\sum_{i=1}^p\I{n-2p>\rem{a}_i} \leq 2p$, which we may rewrite into $n-3p \leq \sum_{i=1}^p\I{n-2p>\rem{a}_i}$. Observe that $\I{n-2p>\rem{a}_i} = 1 - \I{n-2p \leq \rem{a}_i}$, so the required inequality becomes $n-3p \leq \sum_{i=1}^p(1 - \I{n-2p \leq \rem{a}_i})$, or $4p-n \geq \sum_{i=1}^p\I{n-2p \leq \rem{a}_i}$. If $n \leq 3p$, this inequality is obviously true, so we need only consider $3p < n < 4p$, in which case $0 \leq 4p-n < p$. By condition (1), there are at most $4p-n$ values of $i$ with $\rem{a}_i \geq 2p-(4p-n) = n-2p$, which proves the desired inequality.
\qed

\tref{ceiling} gives a complete characterization of 2-ary
meta-Fibonacci recurrences that are formally satisfied by the sequence
$\cln{n}{\alpha}$.\footnote{Note that \tref{ceiling} actually shows formal satisfaction even when we do not require positivity of the parameters.}  As discussed above, this is not the same as generating this sequence as the solution sequence from the recurrence, since in theory the recurrence relation might refer to future terms. The following corollary rules out this possibility:

\begin{corollary}\label{cor:produceceiling}
 Suppose the $2$-ary, order $p$ meta-Fibonacci recurrence relation
  \begin{align*}
    H(n) = H\left(n-s - \sum_{i=1}^p H(n-a_i) \right) + H\left(n-t -
      \sum_{i=1}^p H(n-b_i) \right)
  \end{align*}
  has parameters satisfying
  \begin{enumerate}
  \item For each integer $j$ in $\{0,1,\ldots,p-1\}$, at most $j$ of
    the $\rem{a}_i$s satisfy $\rem{a}_i \le j$ and at most $j$ of them satisfy
    $\rem{a}_i \ge 2p - j$.
  \item For each integer $j$ in $\{0,1,\ldots,p-1\}$, at most $j$ of
    the $\rem{b}_i$s satisfy $\rem{b}_i \le j$ and at most $j$ of them satisfy
    $\rem{b}_i \ge 2p - j$.
  \item There exists an integer $d$ such that either $-s + \sum_{i=1}^p
    \quo{a}_i = 2pd$ and $-t + \sum_{i=1}^p \quo{b}_i = -2pd -p$, or $-s +
    \sum_{i=1}^p \quo{a}_i = -2pd-p$ and $-t + \sum_{i=1}^p \quo{b}_i = 2pd$.
  \end{enumerate}
Furthermore, define initial conditions $H(v) = \cln{v}{2p}$ for $1\leq v \leq c$ where $c = \max\{2p+2s, 2p+2t, a_1, \ldots, a_p, b_1, \ldots, b_p\}$. Then $\cln{n}{2p}$ is the unique solution sequence generated by the recursion $H(n)$. 

\end{corollary}
\begin{proof}
We proceed inductively. The initial conditions cover the base case. Suppose $H(n) = \cln{n}{2p}$ for all $n$ with $0<n<N$, where $N>c$ (since we have $c$ initial conditions). We will show $H(N)$ is well defined by $H(N) = H\left(N-s - \sum_{i=1}^p H(N-a_i) \right) + H\left(N-t -\sum_{i=1}^p H(N-b_i) \right)$ and equals $\cln{N}{2p}$. Since $0<N-a_i<N$ and $0<N-b_i<N$ for all $i$, we may substitute to get $H\left(N-s - \sum_{i=1}^p H(N-a_i) \right) + H\left(N-t -\sum_{i=1}^p H(N-b_i) \right) = H\left(N-s - \sum_{i=1}^p \cln{N-a_i}{2p} \right) + H\left(N-t -\sum_{i=1}^p \cln{N-b_i}{2p} \right)$. 

Now note that $N-s - \sum_{i=1}^p \cln{N-a_i}{2p} < N$, since $s$ is nonnegative and each $\cln{N-a_i}{2p} \geq 1$. Further $N-s - \sum_{i=1}^p \cln{N-a_i}{2p} > N-s-p\cln{N}{2p} > N-s-p(\frac{N}{2p}+1) = N - s -\frac{N}{2} - p = \frac{N}{2} - s - p$, and by assumption $N>2s+2p$, so $N-s - \sum_{i=1}^p \cln{N-a_i}{2p} > 0$. Therefore, by our inductive assumption, $H\left(N-s - \sum_{i=1}^p \cln{N-a_i}{2p} \right) = \cln{N-s - \sum_{i=1}^p \cln{N-a_i}{2p}}{2p}$. A similar argument holds for the other summand, so we have that $H(N) = \cln{N-s - \sum_{i=1}^p \cln{N-a_i}{2p}}{2p} + \cln{N-t - \sum_{i=1}^p \cln{N-b_i}{2p}}{2p}$. But by Theorem $\ref{thm:ceiling}$, we have that $\cln{n}{2p}$ formally satisfies $H(n)$, therefore $\cln{N-s - \sum_{i=1}^p \cln{N-a_i}{2p}}{2p} + \cln{N-t - \sum_{i=1}^p \cln{N-b_i}{2p}}{2p} = \cln{N}{2p} = H(N)$.
\end{proof}

We now present two special cases of \tref{ceiling} for $p=1$ and $p=2$, the first of which was
given as \cref{ceiling1}.  If $p=1$ (writing $a_1=a$ and $b_1=b$) then conditions (1) and (2) of \tref{ceiling}
say that $a$ and $b$ are odd while condition (3) says, up to the usual symmetry
in the parameters, that for some integer $d$, $\flr{a}{2} - s = 2d$ and $\flr{b}{2}-t =
-2d-1$.  If we add the equations in (3) and multiply by $2$ we get $2\flr{a}{2} - 2s + 2\flr{b}{2} - 2t = -2$, from which we derive the result promised in Corollary \ref{cor:ceiling1}, namely, $2(s+t) = a + b$, with $a$ and $b$ both
odd.

If $p=2$ then the conditions of \tref{ceiling} can be simplified as follows:

\begin{corollary}
  Let $H(n) = H(n-s-H(n-a)-H(n-b)) + H(n-t-H(n-c)-H(n-d))$ be an order 2
 nested recurrence relation.  The sequence $\cln{n}{4}$ is a
  solution to $H(n)$ if and only if there is an odd integer $\kappa$ such
  that the following conditions are satisfied:
\begin{enumerate}
\item[(i)] $a+b \in \set{ 4(s + \kappa)-1, 4(s + \kappa), 4(s + \kappa)+1 }$,
\item[(ii)] $c+d \in \set{ 4(t - \kappa)-1, 4(t - \kappa), 4(t - \kappa)+1 }$,
\item[(iii)] $a,b,c,d \not\equiv 0 \bmod{4}$.
\end{enumerate}
\end{corollary}
\begin{proof}
  Suppose $\cln{n}{4}$ is a solution to $H(n)$.  We apply
  \tref{ceiling} with $p = 2$.  Condition (3) of \tref{ceiling} says,
  up to switching the roles of $\set{a,b,s}$ and $\set{c,d,t}$, that
  $\lfloor \frac{a}{4} \rfloor + \lfloor \frac{b}{4} \rfloor = s+4e$ and $\lfloor \frac{c}{4} \rfloor +\lfloor \frac{d}{4} \rfloor =
  t-4e-2$.  Multiplying both sides by $4$ and adding $\rem{a}+\rem{b}$
  gives $a+b = 4(s+4e) + \rem{a} + \rem{b}$. Similarly, we have $c+d = 4(t-4e-2) + \rem{c} + \rem{d}$.  

First, note that using conditions (1) and (2) of \tref{ceiling} with $j=0$ establishes (iii), since it shows that none of $\rem{a},\rem{b},\rem{c},\rem{d}$ can be 0. Using conditions (1) and (2) of \tref{ceiling} with $j=1$ shows that $\rem{a}$ and $\rem{b}$ cannot both be 1 and cannot both be 3, and the same for $\rem{c}$ and $\rem{d}$. Combining these two facts, $\rem{a}+\rem{b} \in \set{3,4,5}$ and similarly $\rem{c}+\rem{d} \in \set{3,4,5}$. 

So $a+b = 4(s+4e+1) + \rem{a}  +\rem{b} - 4$, and $c+d = 4(t-4e-2+1) +\rem{c} + \rem{d}-4$.  Since  $4e+1 = -(-4e-2+1)$, we may set $\kappa = 4e+1$.

  The argument can be reversed so that (1), (2) and (3) follow from
  (i), (ii) and (iii).  That is, assume $\kappa$ is an odd integer.  If $\kappa
  = 4e+1$ for some integer $e$, then by (i), we have $a+b = 4(s+4e) + E$
  where $E\in\set{3,4,5}$.  Subtracting $\rem{a} + \rem{b}$ and
  dividing by $4$ gives $\lfloor \frac{a}{4} \rfloor + \lfloor \frac{b}{4} \rfloor = s+4e +
  \frac{E-\rem{a}-\rem{b}}{4}$.  But the left hand side is an integer,
  so $E-\rem{a}-\rem{b} \pmod{4} = 0$. Since $a,b \not\equiv 0 \bmod{4}$, we have $1 \leq \rem{a} \leq 3$, and $1 \leq \rem{b} \leq 3$, so $2\leq \rem{a}+\rem{b}\le
  6$. Therefore, we must have $E= \rem{a} + \rem{b}$.  This forbids
  $\rem{a}=\rem{b} = 1$ and $\rem{a}=\rem{b} = 3$, which shows condition (1) for $j=1$.  Combining this
  with condition (iii) which shows condition (1) for $j=0$, we establish condition (1). A similar argument yields condition (2). Furthermore, since we showed $E= \rem{a}+\rem{b}$, and  $\lfloor \frac{a}{4} \rfloor + \lfloor \frac{b}{4} \rfloor = s+4e +
  \frac{E-\rem{a}-\rem{b}}{4}$, we have $\lfloor \frac{a}{4} \rfloor + \lfloor \frac{b}{4} \rfloor = s+4e$, establishing the first part of condition (3); a similar argument on $c$ and $d$ yields the second part of condition (3). By \tref{ceiling}, therefore, $\cln{n}{4}$ is a solution to $H(n)$. 

A similar argument holds if $\kappa = 4e+3$, where we begin by invoking condition (ii) to get $c+d=4(t-4e-3) + E-4$ and rearrange to $c+d=4(t-4(e+1))+E$, then proceed as above. 
\end{proof}

\section{Concluding Remarks} \label{sec:Conc}

We have an almost complete characterization of 2-ary order 1 Conolly-like recursions. The only missing component is a proof that $\SEQ{0}{1}{1}{2}$ and $\SEQ{0}{2}{3}{5}$ are the only two order 1 recursions satisfied by the Conolly sequence. To show this, we would like an analogue of Corollary \ref{cor:ceiling1} for $(0,1)$-Conolly recursions.

For 2-ary recursions of higher order we have shown that $(\alpha, \beta)$-Conolly recursions exist for all permissible pairs $(\alpha, \beta)$. Unlike the situation for order 1, we believe that for $\beta>0$ there are many $(\alpha, \beta)$-Conolly recursions. To expand on the existence result proved in Theorem \ref{thm:alphabeta} we would like an analogue to Theorem \ref{thm:ceiling}; a starting point would be to find a way to show that all of the 180 recursions listed in Conjecture \ref{conj:order2} are indeed Conolly-like.

Nothing is currently known about the existence of Conolly-like recursions with arity higher than 2. An empirical investigation along the lines described in \sref{Exp} likely would be a useful starting point.

Another direction for further work is that \tref{ceiling} seems
adaptable to $k$-ary recurrence relations of order $p$ and arbitrary
$\alpha$.  A sequence $A(n)$ that satisfies $\cln{n}{\alpha}$ also
satisfies the limit $\frac{A(n)}{n} \longrightarrow \frac{1}{\alpha}$.
By \tref{ordlimit}, $\alpha = \frac{kp}{k-1}$ and since
$\gcd(k,k-1)=1$, we must have $p=q(k-1)$ for some $q$.  The function
$\cln{n}{kq}$ can be expressed as a sum of $k$ ceiling functions,
analogous to
\begin{align*}
\cln{n}{2p} = \cln{n}{4p} + \cln{n-2p}{4p},
\end{align*}
using repeated applications of Exercise 35 of
Section 1.2.4 in \cite{knuthvol1}, and the appropriate substitution into Equation 3.24 of
\cite{GKP}.  The result is 
\begin{align}\label{eq:ceil}
  \cln{n}{kq} = \cln{n}{k^2q} + \cln{n-kq}{k^2q} + \cln{n-2kq}{k^2q} +
  \cdots +\cln{n-(k-1)kq}{k^2q},
\end{align}
and following the pattern of \tref{ceiling}, one would want to use the
fact that to show that the ith term on the right hand side corresponds
to the ith term in the sum $H(n) = \sum_{i=1}^k h_i(n)$ for some
$k$-ary recurrence relation $H$ of order $p$.  In this case one would
define $\quo{m}$ and $\rem{m}$ to satisfy $m=kq\quo{m} + \rem{m}$ for
$0\le\rem{m}<kq$ so as to take advantage of the Iversonians used in
\tref{ceiling}.  Particular attention should be paid to showing that
the form in Equation (\ref{eq:ceil}) is necessary, given $k$ and $q$.

Further consideration might be given to sequences of the form
$\cln{rn}{q}$ with positive integers $r$ and $q$ such that $0<r/q\le
1$ and $\gcd(r,q) =1$.  These are only Conolly-like for $r=1$, but
nevertheless they have some interesting features.  For example, this
is a necessary and sufficient condition for a slow-growing sequence to
have a periodic frequency function.  These sequences will be the subject of a forthcoming publication.

Finally, it would be interesting to apply the tree technique to help identify recursions with solutions whose frequency functions are linear combinations of functions other than 1 and $r_m$. This will be the subject of a forthcoming publication.

\end{document}

%% file: figureXXXtruncate.pstex_t
\begin{picture}(0,0)%
\includegraphics{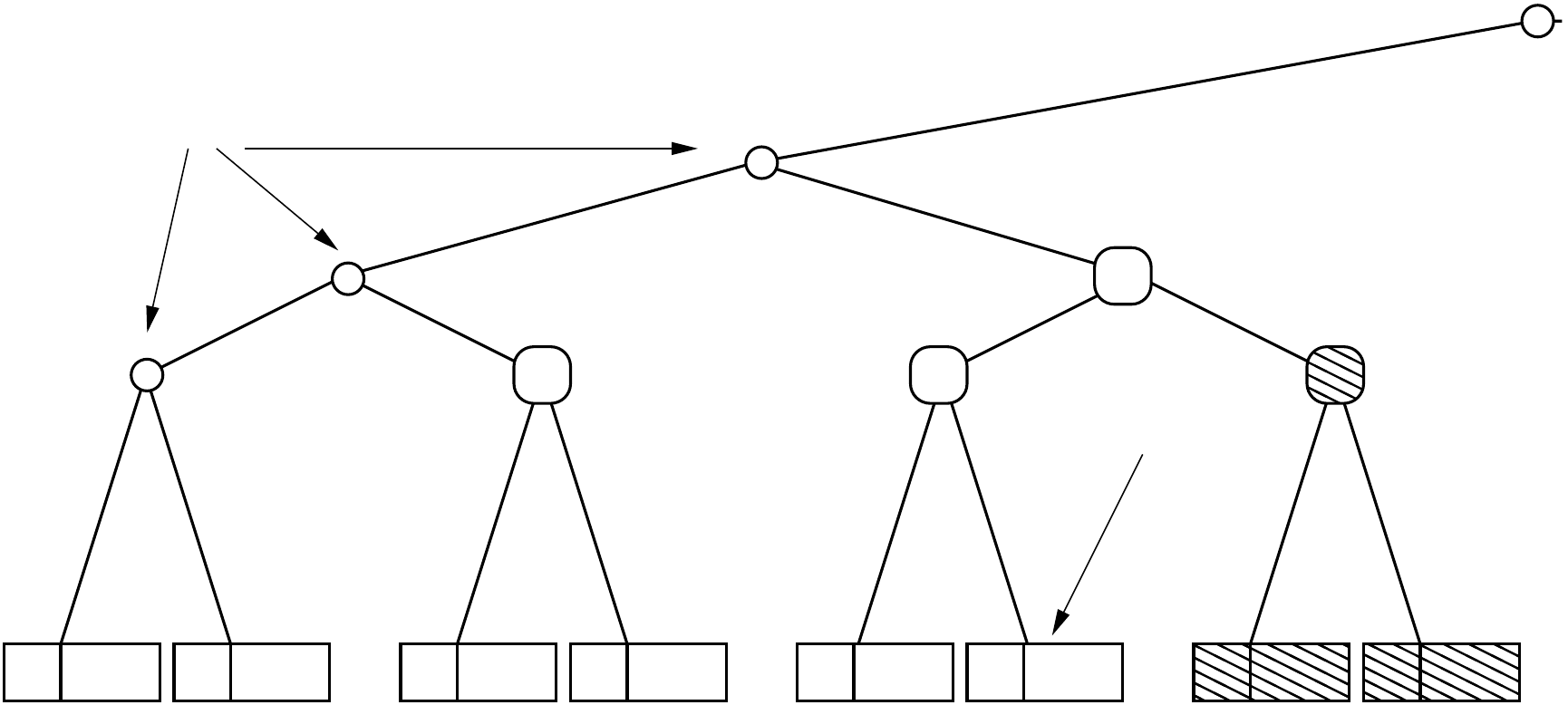}%
\end{picture}%
\setlength{\unitlength}{3947sp}%
\begingroup\makeatletter\ifx\SetFigFont\undefined%
\gdef\SetFigFont#1#2#3#4#5{%
  \reset@font\fontsize{#1}{#2pt}%
  \fontfamily{#3}\fontseries{#4}\fontshape{#5}%
  \selectfont}%
\fi\endgroup%
\begin{picture}(8294,3720)(5004,-6608)
\put(10201,-5161){\makebox(0,0)[lb]{\smash{{\SetFigFont{12}{14.4}{\familydefault}{\mddefault}{\updefault}{\color[rgb]{0,0,0}last label in $T(20)$}%
}}}}
\put(5101,-6511){\makebox(0,0)[lb]{\smash{{\SetFigFont{12}{14.4}{\familydefault}{\mddefault}{\updefault}$1$}}}}
\put(9301,-6511){\makebox(0,0)[lb]{\smash{{\SetFigFont{12}{14.4}{\familydefault}{\mddefault}{\updefault}$16$}}}}
\put(10201,-6511){\makebox(0,0)[lb]{\smash{{\SetFigFont{12}{14.4}{\familydefault}{\mddefault}{\updefault}$19$}}}}
\put(5401,-6511){\makebox(0,0)[lb]{\smash{{\SetFigFont{12}{14.4}{\familydefault}{\mddefault}{\updefault}$2,3$}}}}
\put(6001,-6511){\makebox(0,0)[lb]{\smash{{\SetFigFont{12}{14.4}{\familydefault}{\mddefault}{\updefault}$4$}}}}
\put(6301,-6511){\makebox(0,0)[lb]{\smash{{\SetFigFont{12}{14.4}{\familydefault}{\mddefault}{\updefault}$5,6$}}}}
\put(7201,-6511){\makebox(0,0)[lb]{\smash{{\SetFigFont{12}{14.4}{\familydefault}{\mddefault}{\updefault}$8$}}}}
\put(7501,-6511){\makebox(0,0)[lb]{\smash{{\SetFigFont{12}{14.4}{\familydefault}{\mddefault}{\updefault}$9,10$}}}}
\put(8101,-6511){\makebox(0,0)[lb]{\smash{{\SetFigFont{12}{14.4}{\familydefault}{\mddefault}{\updefault}$11$}}}}
\put(10501,-6511){\makebox(0,0)[lb]{\smash{{\SetFigFont{12}{14.4}{\familydefault}{\mddefault}{\updefault}$20$}}}}
\put(8326,-6511){\makebox(0,0)[lb]{\smash{{\SetFigFont{12}{14.4}{\familydefault}{\mddefault}{\updefault}$12,13$}}}}
\put(9526,-6511){\makebox(0,0)[lb]{\smash{{\SetFigFont{12}{14.4}{\familydefault}{\mddefault}{\updefault}$17,18$}}}}
\put(5026,-3511){\makebox(0,0)[lb]{\smash{{\SetFigFont{12}{14.4}{\familydefault}{\mddefault}{\updefault}{\color[rgb]{0,0,0}first, second and third $s$-nodes}%
}}}}
\put(7801,-4936){\makebox(0,0)[lb]{\smash{{\SetFigFont{12}{14.4}{\familydefault}{\mddefault}{\updefault}$7$}}}}
\put(9901,-4936){\makebox(0,0)[lb]{\smash{{\SetFigFont{12}{14.4}{\familydefault}{\mddefault}{\updefault}$15$}}}}
\put(10876,-4411){\makebox(0,0)[lb]{\smash{{\SetFigFont{12}{14.4}{\familydefault}{\mddefault}{\updefault}$14$}}}}
\end{picture}%

%% file: prunetrunc.pstex_t
\begin{picture}(0,0)%
\includegraphics{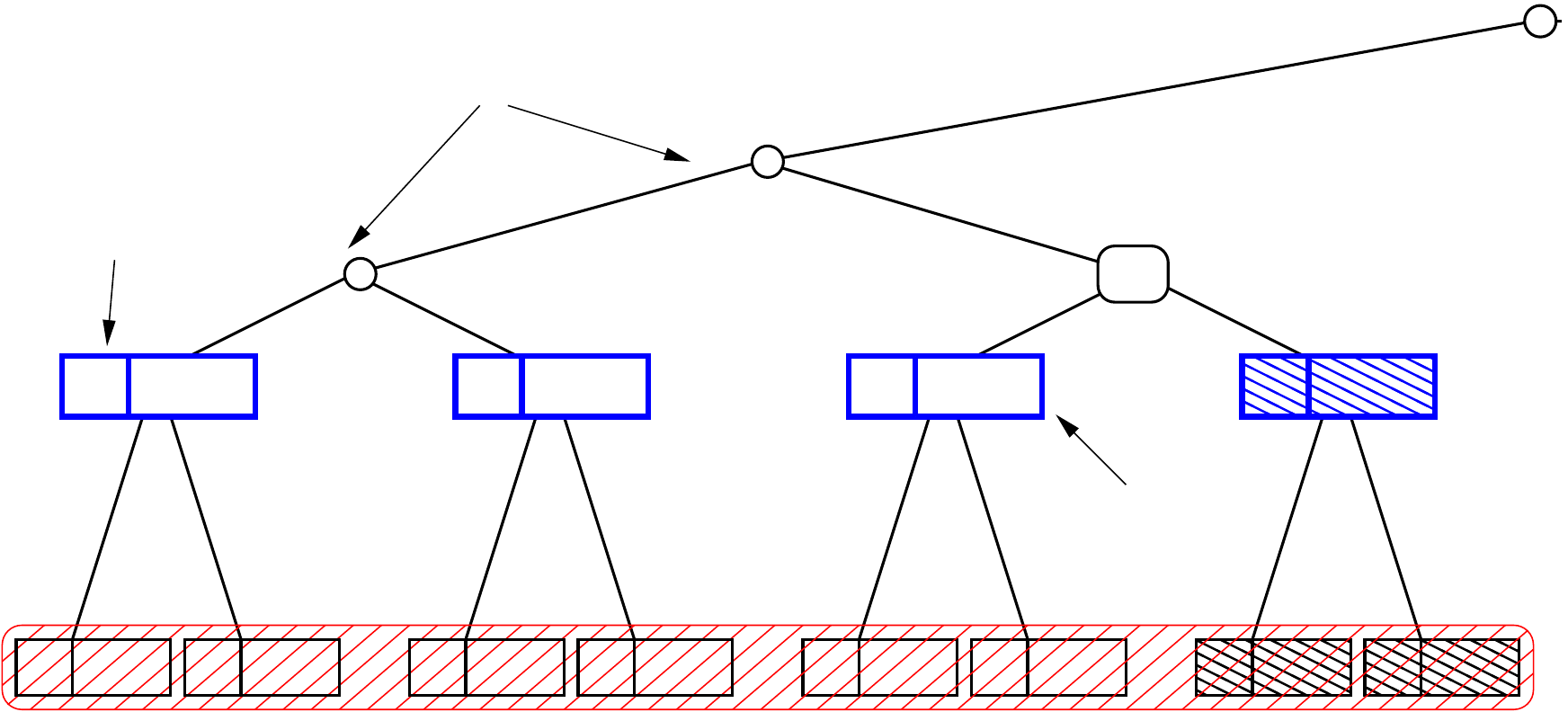}%
\end{picture}%
\setlength{\unitlength}{3947sp}%
\begingroup\makeatletter\ifx\SetFigFont\undefined%
\gdef\SetFigFont#1#2#3#4#5{%
  \reset@font\fontsize{#1}{#2pt}%
  \fontfamily{#3}\fontseries{#4}\fontshape{#5}%
  \selectfont}%
\fi\endgroup%
\begin{picture}(8359,3785)(4939,-6673)
\put(10201,-5011){\makebox(0,0)[lb]{\smash{{\SetFigFont{12}{14.4}{\familydefault}{\mddefault}{\updefault}{\color[rgb]{0,0,1}$\mathbf{19}$}%
}}}}
\put(6331,-6511){\makebox(0,0)[lb]{\smash{{\SetFigFont{12}{14.4}{\familydefault}{\mddefault}{\updefault}{\color[rgb]{0,0,0}$5,6$}%
}}}}
\put(7201,-6511){\makebox(0,0)[lb]{\smash{{\SetFigFont{12}{14.4}{\familydefault}{\mddefault}{\updefault}{\color[rgb]{0,0,0}$8$}%
}}}}
\put(7501,-6511){\makebox(0,0)[lb]{\smash{{\SetFigFont{12}{14.4}{\familydefault}{\mddefault}{\updefault}{\color[rgb]{0,0,0}$9,10$}%
}}}}
\put(8101,-6511){\makebox(0,0)[lb]{\smash{{\SetFigFont{12}{14.4}{\familydefault}{\mddefault}{\updefault}{\color[rgb]{0,0,0}$11$}%
}}}}
\put(8348,-6503){\makebox(0,0)[lb]{\smash{{\SetFigFont{12}{14.4}{\familydefault}{\mddefault}{\updefault}{\color[rgb]{0,0,0}$12,13$}%
}}}}
\put(9301,-6511){\makebox(0,0)[lb]{\smash{{\SetFigFont{12}{14.4}{\familydefault}{\mddefault}{\updefault}{\color[rgb]{0,0,0}$16$}%
}}}}
\put(9571,-6511){\makebox(0,0)[lb]{\smash{{\SetFigFont{12}{14.4}{\familydefault}{\mddefault}{\updefault}{\color[rgb]{0,0,0}$17,18$}%
}}}}
\put(10164,-6511){\makebox(0,0)[lb]{\smash{{\SetFigFont{12}{14.4}{\familydefault}{\mddefault}{\updefault}{\color[rgb]{0,0,0}$19$}%
}}}}
\put(10471,-6511){\makebox(0,0)[lb]{\smash{{\SetFigFont{12}{14.4}{\familydefault}{\mddefault}{\updefault}{\color[rgb]{0,0,0}$20$}%
}}}}
\put(10501,-5686){\makebox(0,0)[lb]{\smash{{\SetFigFont{12}{14.4}{\familydefault}{\mddefault}{\updefault}{\color[rgb]{0,0,1}added label}%
}}}}
\put(10501,-5911){\makebox(0,0)[lb]{\smash{{\SetFigFont{12}{14.4}{\familydefault}{\mddefault}{\updefault}{\color[rgb]{0,0,1}after mapping}%
}}}}
\put(5401,-5011){\makebox(0,0)[lb]{\smash{{\SetFigFont{12}{14.4}{\familydefault}{\mddefault}{\updefault}{\color[rgb]{0,0,1}$\mathbf{0}$}%
}}}}
\put(5851,-5011){\makebox(0,0)[lb]{\smash{{\SetFigFont{12}{14.4}{\familydefault}{\mddefault}{\updefault}{\color[rgb]{0,0,0}$3,6$}%
}}}}
\put(7501,-5011){\makebox(0,0)[lb]{\smash{{\SetFigFont{12}{14.4}{\familydefault}{\mddefault}{\updefault}{\color[rgb]{0,0,0}$7$}%
}}}}
\put(7801,-5011){\makebox(0,0)[lb]{\smash{{\SetFigFont{12}{14.4}{\familydefault}{\mddefault}{\updefault}{\color[rgb]{0,0,0}$10,13$}%
}}}}
\put(9526,-5011){\makebox(0,0)[lb]{\smash{{\SetFigFont{12}{14.4}{\familydefault}{\mddefault}{\updefault}{\color[rgb]{0,0,0}$15$}%
}}}}
\put(9901,-5011){\makebox(0,0)[lb]{\smash{{\SetFigFont{12}{14.4}{\familydefault}{\mddefault}{\updefault}{\color[rgb]{0,0,0}$18,$}%
}}}}
\put(5116,-6503){\makebox(0,0)[lb]{\smash{{\SetFigFont{12}{14.4}{\familydefault}{\mddefault}{\updefault}{\color[rgb]{0,0,0}$1$}%
}}}}
\put(6016,-6511){\makebox(0,0)[lb]{\smash{{\SetFigFont{12}{14.4}{\familydefault}{\mddefault}{\updefault}{\color[rgb]{0,0,0}$4$}%
}}}}
\put(5439,-6511){\makebox(0,0)[lb]{\smash{{\SetFigFont{12}{14.4}{\familydefault}{\mddefault}{\updefault}{\color[rgb]{0,0,0}$2,3$}%
}}}}
\put(6376,-3361){\makebox(0,0)[lb]{\smash{{\SetFigFont{12}{14.4}{\familydefault}{\mddefault}{\updefault}{\color[rgb]{0,0,1}new first and second $s$-nodes}%
}}}}
\put(10876,-4411){\makebox(0,0)[lb]{\smash{{\SetFigFont{12}{14.4}{\familydefault}{\mddefault}{\updefault}{\color[rgb]{0,0,0}$14$}%
}}}}
\put(5026,-4036){\makebox(0,0)[lb]{\smash{{\SetFigFont{12}{14.4}{\familydefault}{\mddefault}{\updefault}{\color[rgb]{0,0,1}labeled $0$}%
}}}}
\put(5026,-4261){\makebox(0,0)[lb]{\smash{{\SetFigFont{12}{14.4}{\familydefault}{\mddefault}{\updefault}{\color[rgb]{0,0,1}by the initial step}%
}}}}
\end{picture}%

%% file: Uofn.pstex_t
\begin{picture}(0,0)%
\includegraphics{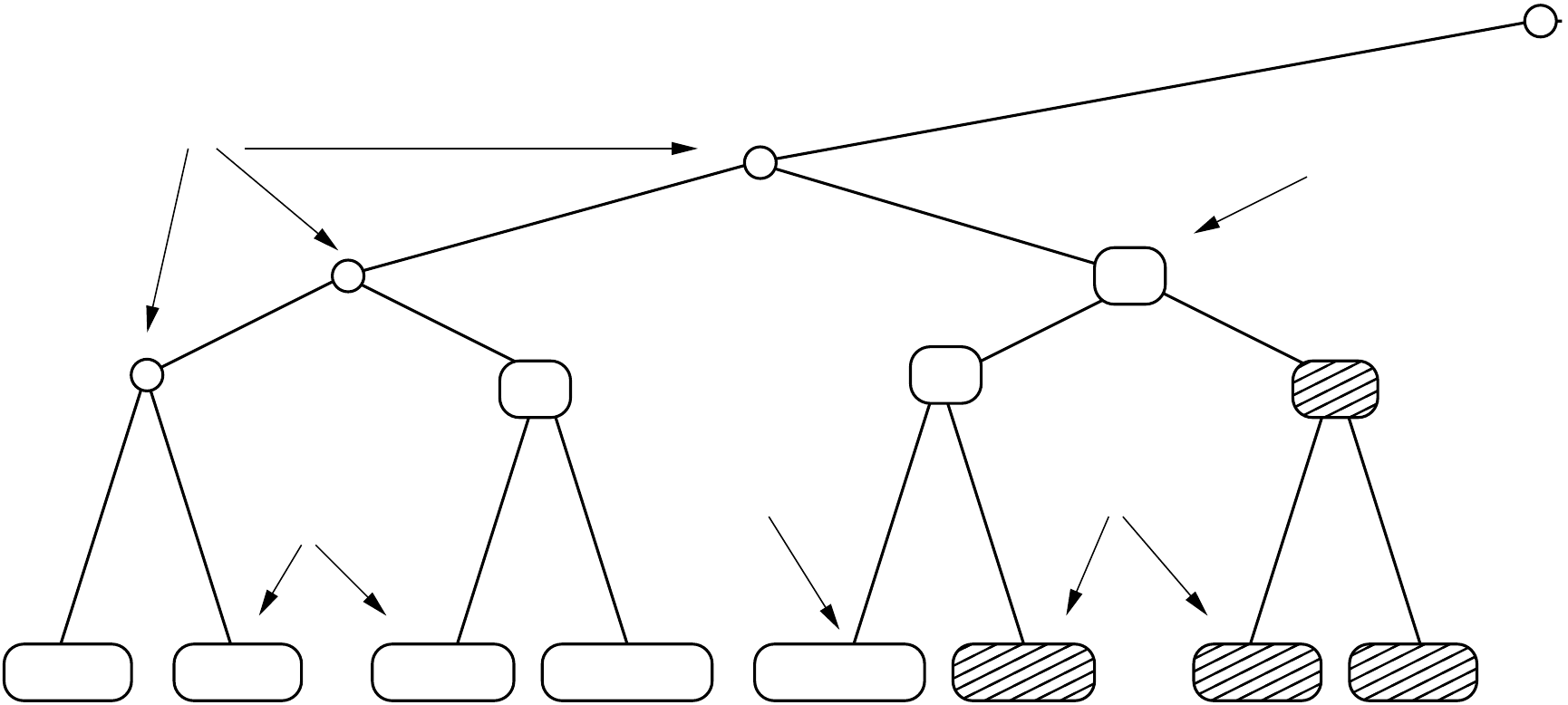}%
\end{picture}%
\setlength{\unitlength}{3947sp}%
\begingroup\makeatletter\ifx\SetFigFont\undefined%
\gdef\SetFigFont#1#2#3#4#5{%
  \reset@font\fontsize{#1}{#2pt}%
  \fontfamily{#3}\fontseries{#4}\fontshape{#5}%
  \selectfont}%
\fi\endgroup%
\begin{picture}(8294,3720)(5004,-6608)
\put(5026,-3511){\makebox(0,0)[lb]{\smash{{\SetFigFont{12}{14.4}{\familydefault}{\mddefault}{\updefault}{\color[rgb]{0,0,0}first, second and third $s$-nodes}%
}}}}
\put(9076,-6511){\makebox(0,0)[lb]{\smash{{\SetFigFont{12}{14.4}{\familydefault}{\mddefault}{\updefault}{\color[rgb]{0,0,0}$16,17$}%
}}}}
\put(9901,-4936){\makebox(0,0)[lb]{\smash{{\SetFigFont{12}{14.4}{\familydefault}{\mddefault}{\updefault}$15$}}}}
\put(10876,-4411){\makebox(0,0)[lb]{\smash{{\SetFigFont{12}{14.4}{\familydefault}{\mddefault}{\updefault}$14$}}}}
\put(5101,-6511){\makebox(0,0)[lb]{\smash{{\SetFigFont{12}{14.4}{\familydefault}{\mddefault}{\updefault}$1,2,3$}}}}
\put(8701,-5461){\makebox(0,0)[lb]{\smash{{\SetFigFont{12}{14.4}{\familydefault}{\mddefault}{\updefault}{\color[rgb]{0,0,0}last label}%
}}}}
\put(6001,-6511){\makebox(0,0)[lb]{\smash{{\SetFigFont{12}{14.4}{\familydefault}{\mddefault}{\updefault}{\color[rgb]{0,0,0}$4,5,6$}%
}}}}
\put(7051,-6511){\makebox(0,0)[lb]{\smash{{\SetFigFont{12}{14.4}{\familydefault}{\mddefault}{\updefault}{\color[rgb]{0,0,0}$8,9,10$}%
}}}}
\put(7951,-6511){\makebox(0,0)[lb]{\smash{{\SetFigFont{12}{14.4}{\familydefault}{\mddefault}{\updefault}{\color[rgb]{0,0,0}$11,12,13$}%
}}}}
\put(7801,-5011){\makebox(0,0)[lb]{\smash{{\SetFigFont{12}{14.4}{\familydefault}{\mddefault}{\updefault}$7$}}}}
\put(7801,-4486){\makebox(0,0)[lb]{\smash{{\SetFigFont{12}{14.4}{\familydefault}{\mddefault}{\updefault}{\color[rgb]{0,0,0}$U(17)$ for $\beta=1,\alpha = 2$}%
}}}}
\put(6151,-5671){\makebox(0,0)[lb]{\smash{{\SetFigFont{12}{14.4}{\familydefault}{\mddefault}{\updefault}{\color[rgb]{0,0,0}$\alpha+\beta$ labels}%
}}}}
\put(10501,-3991){\makebox(0,0)[lb]{\smash{{\SetFigFont{12}{14.4}{\familydefault}{\mddefault}{\updefault}{\color[rgb]{0,0,0}labels}%
}}}}
\put(6151,-5416){\makebox(0,0)[lb]{\smash{{\SetFigFont{12}{14.4}{\familydefault}{\mddefault}{\updefault}{\color[rgb]{0,0,0}leaves with}%
}}}}
\put(10501,-3736){\makebox(0,0)[lb]{\smash{{\SetFigFont{12}{14.4}{\familydefault}{\mddefault}{\updefault}{\color[rgb]{0,0,0}regular nodes with $\beta$}%
}}}}
\put(10576,-5311){\makebox(0,0)[lb]{\smash{{\SetFigFont{12}{14.4}{\familydefault}{\mddefault}{\updefault}{\color[rgb]{0,0,0}empty}%
}}}}
\put(10576,-5566){\makebox(0,0)[lb]{\smash{{\SetFigFont{12}{14.4}{\familydefault}{\mddefault}{\updefault}{\color[rgb]{0,0,0}nodes}%
}}}}
\end{picture}%

%% file: uofnprune.pstex_t
\begin{picture}(0,0)%
\includegraphics{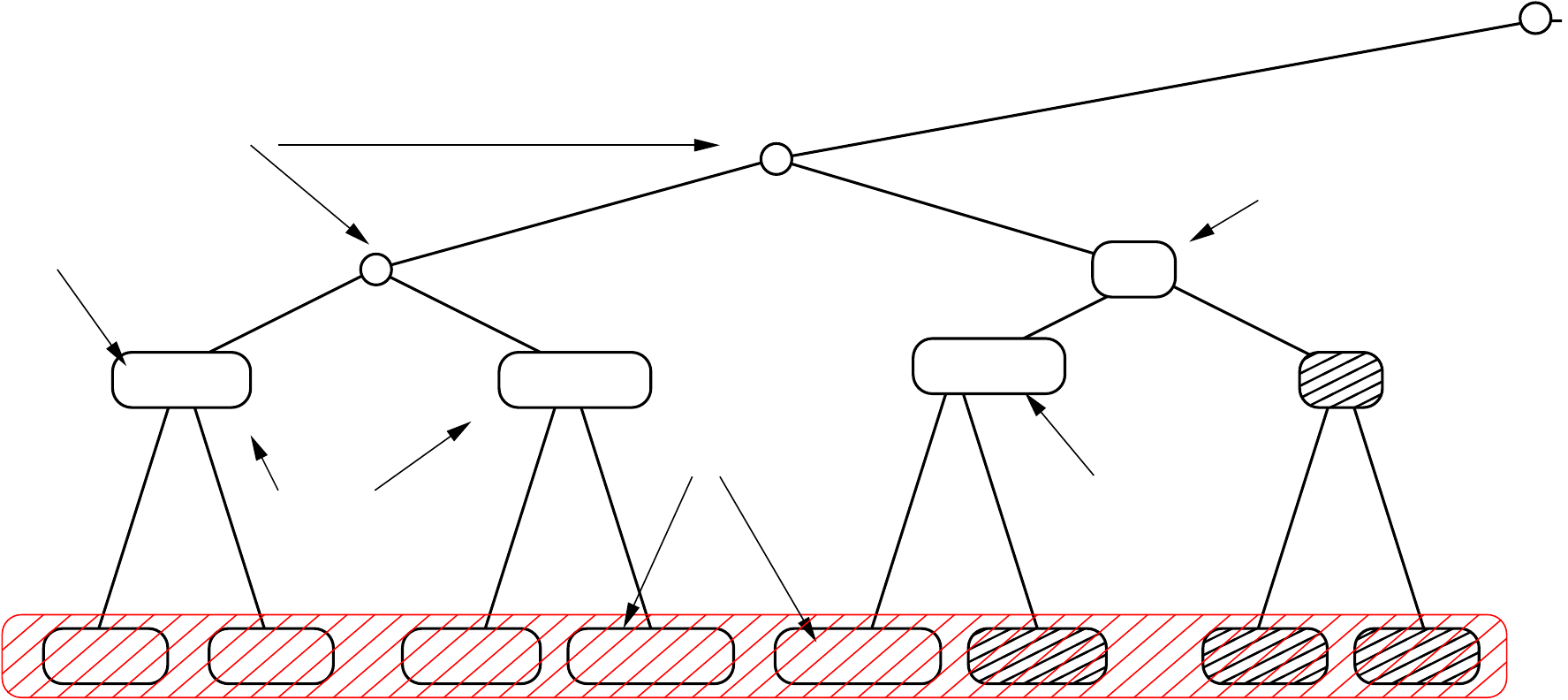}%
\end{picture}%
\setlength{\unitlength}{3947sp}%
\begingroup\makeatletter\ifx\SetFigFont\undefined%
\gdef\SetFigFont#1#2#3#4#5{%
  \reset@font\fontsize{#1}{#2pt}%
  \fontfamily{#3}\fontseries{#4}\fontshape{#5}%
  \selectfont}%
\fi\endgroup%
\begin{picture}(8509,3800)(4789,-6673)
\put(5026,-3511){\makebox(0,0)[lb]{\smash{{\SetFigFont{12}{14.4}{\familydefault}{\mddefault}{\updefault}{\color[rgb]{0,0,0}first and second $s$-nodes of pruned tree}%
}}}}
\put(10426,-6061){\makebox(0,0)[lb]{\smash{{\SetFigFont{12}{14.4}{\familydefault}{\mddefault}{\updefault}{\color[rgb]{0,0,0}$\beta$ labels}%
}}}}
\put(9826,-4936){\makebox(0,0)[lb]{\smash{{\SetFigFont{12}{14.4}{\familydefault}{\mddefault}{\updefault}$15,\xout{17}$}}}}
\put(7501,-4486){\makebox(0,0)[lb]{\smash{{\SetFigFont{12}{14.4}{\familydefault}{\mddefault}{\updefault}{\color[rgb]{0,0,0}pruned $U(17)$ for $\beta=1,\alpha = 2$}%
}}}}
\put(7576,-5011){\makebox(0,0)[lb]{\smash{{\SetFigFont{12}{14.4}{\familydefault}{\mddefault}{\updefault}$7,10,13$}}}}
\put(10801,-4411){\makebox(0,0)[lb]{\smash{{\SetFigFont{12}{14.4}{\familydefault}{\mddefault}{\updefault}$14$}}}}
\put(5101,-6511){\makebox(0,0)[lb]{\smash{{\SetFigFont{12}{14.4}{\familydefault}{\mddefault}{\updefault}$\xout{1,2},3$}}}}
\put(6001,-6511){\makebox(0,0)[lb]{\smash{{\SetFigFont{12}{14.4}{\familydefault}{\mddefault}{\updefault}{\color[rgb]{0,0,0}$\xout{4,5},6$}%
}}}}
\put(7051,-6511){\makebox(0,0)[lb]{\smash{{\SetFigFont{12}{14.4}{\familydefault}{\mddefault}{\updefault}{\color[rgb]{0,0,0}$\xout{8,9},10$}%
}}}}
\put(7951,-6511){\makebox(0,0)[lb]{\smash{{\SetFigFont{12}{14.4}{\familydefault}{\mddefault}{\updefault}{\color[rgb]{0,0,0}$\xout{11,12},13$}%
}}}}
\put(9076,-6511){\makebox(0,0)[lb]{\smash{{\SetFigFont{12}{14.4}{\familydefault}{\mddefault}{\updefault}{\color[rgb]{0,0,0}$\xout{16},17$}%
}}}}
\put(10351,-5611){\makebox(0,0)[lb]{\smash{{\SetFigFont{12}{14.4}{\familydefault}{\mddefault}{\updefault}{\color[rgb]{0,0,0}correction step}%
}}}}
\put(10351,-5836){\makebox(0,0)[lb]{\smash{{\SetFigFont{12}{14.4}{\familydefault}{\mddefault}{\updefault}{\color[rgb]{0,0,0}delete last}%
}}}}
\put(5476,-5011){\makebox(0,0)[lb]{\smash{{\SetFigFont{12}{14.4}{\familydefault}{\mddefault}{\updefault}{\color[rgb]{0,0,0}$\mathbf{18},3,6$}%
}}}}
\put(4876,-4261){\makebox(0,0)[lb]{\smash{{\SetFigFont{12}{14.4}{\familydefault}{\mddefault}{\updefault}{\color[rgb]{0,0,0}initial step}%
}}}}
\put(6226,-5686){\makebox(0,0)[lb]{\smash{{\SetFigFont{12}{14.4}{\familydefault}{\mddefault}{\updefault}{\color[rgb]{0,0,0}new leaves}%
}}}}
\put(10951,-3886){\makebox(0,0)[lb]{\smash{{\SetFigFont{12}{14.4}{\familydefault}{\mddefault}{\updefault}{\color[rgb]{0,0,0}new regular node}%
}}}}
\put(8326,-5386){\makebox(0,0)[lb]{\smash{{\SetFigFont{12}{14.4}{\familydefault}{\mddefault}{\updefault}{\color[rgb]{0,0,0}deletion step}%
}}}}
\end{picture}%

%% file: Uofnnegalpha.pstex_t
\begin{picture}(0,0)%
\includegraphics{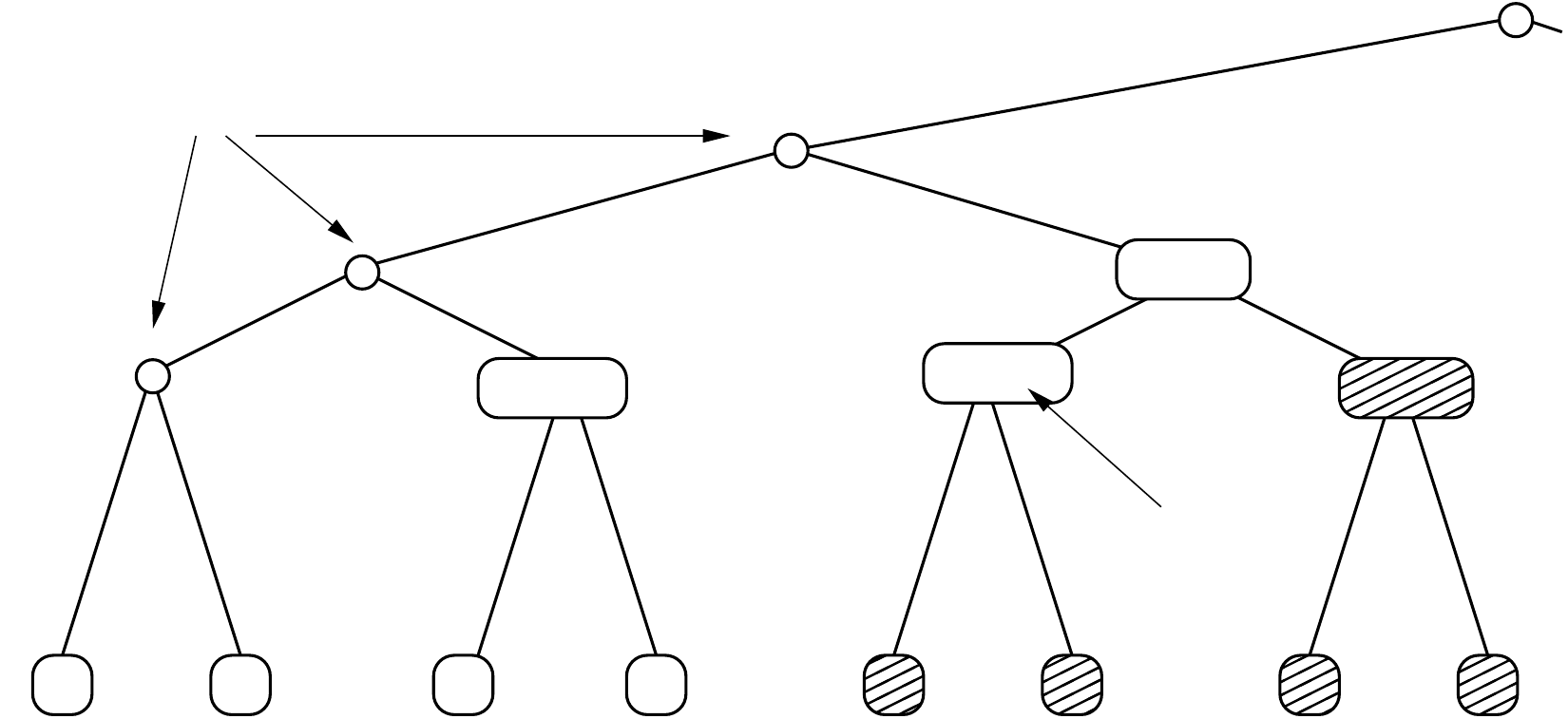}%
\end{picture}%
\setlength{\unitlength}{3947sp}%
\begingroup\makeatletter\ifx\SetFigFont\undefined%
\gdef\SetFigFont#1#2#3#4#5{%
  \reset@font\fontsize{#1}{#2pt}%
  \fontfamily{#3}\fontseries{#4}\fontshape{#5}%
  \selectfont}%
\fi\endgroup%
\begin{picture}(7912,3630)(5011,-6608)
\put(5026,-3511){\makebox(0,0)[lb]{\smash{{\SetFigFont{12}{14.4}{\familydefault}{\mddefault}{\updefault}{\color[rgb]{0,0,0}first, second and third $s$-nodes}%
}}}}
\put(10801,-5761){\makebox(0,0)[lb]{\smash{{\SetFigFont{12}{14.4}{\familydefault}{\mddefault}{\updefault}{\color[rgb]{0,0,0}last label}%
}}}}
\put(7576,-5011){\makebox(0,0)[lb]{\smash{{\SetFigFont{12}{14.4}{\familydefault}{\mddefault}{\updefault}$3,4,5$}}}}
\put(7576,-4486){\makebox(0,0)[lb]{\smash{{\SetFigFont{12}{14.4}{\familydefault}{\mddefault}{\updefault}{\color[rgb]{0,0,0}$U(12)$ for $\beta=3,\alpha = -2$}%
}}}}
\put(10726,-4411){\makebox(0,0)[lb]{\smash{{\SetFigFont{12}{14.4}{\familydefault}{\mddefault}{\updefault}$8,9,10$}}}}
\put(9751,-4936){\makebox(0,0)[lb]{\smash{{\SetFigFont{12}{14.4}{\familydefault}{\mddefault}{\updefault}$11,12$}}}}
\put(5251,-6511){\makebox(0,0)[lb]{\smash{{\SetFigFont{12}{14.4}{\familydefault}{\mddefault}{\updefault}$1$}}}}
\put(6151,-6511){\makebox(0,0)[lb]{\smash{{\SetFigFont{12}{14.4}{\familydefault}{\mddefault}{\updefault}{\color[rgb]{0,0,0}$2$}%
}}}}
\put(7276,-6511){\makebox(0,0)[lb]{\smash{{\SetFigFont{12}{14.4}{\familydefault}{\mddefault}{\updefault}{\color[rgb]{0,0,0}$6$}%
}}}}
\put(8251,-6511){\makebox(0,0)[lb]{\smash{{\SetFigFont{12}{14.4}{\familydefault}{\mddefault}{\updefault}{\color[rgb]{0,0,0}$7$}%
}}}}
\end{picture}%

%% file: Uofnnegalphaprune.pstex_t
\begin{picture}(0,0)%
\includegraphics{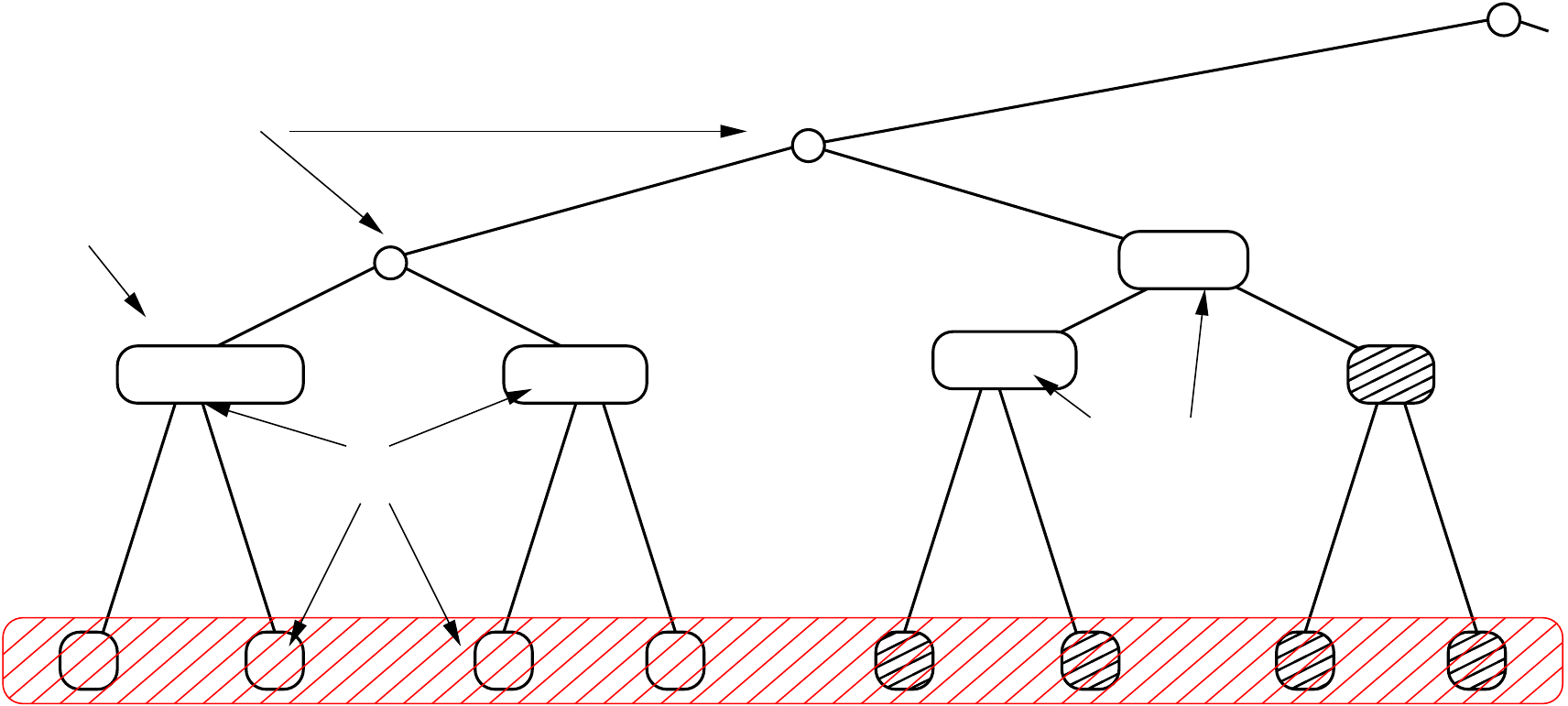}%
\end{picture}%
\setlength{\unitlength}{3947sp}%
\begingroup\makeatletter\ifx\SetFigFont\undefined%
\gdef\SetFigFont#1#2#3#4#5{%
  \reset@font\fontsize{#1}{#2pt}%
  \fontfamily{#3}\fontseries{#4}\fontshape{#5}%
  \selectfont}%
\fi\endgroup%
\begin{picture}(8202,3695)(4786,-6673)
\put(5026,-3511){\makebox(0,0)[lb]{\smash{{\SetFigFont{12}{14.4}{\familydefault}{\mddefault}{\updefault}{\color[rgb]{0,0,0}first and second $s$-nodes of pruned tree}%
}}}}
\put(6151,-6511){\makebox(0,0)[lb]{\smash{{\SetFigFont{12}{14.4}{\familydefault}{\mddefault}{\updefault}{\color[rgb]{0,0,0}$\xout{2}$}%
}}}}
\put(8251,-6511){\makebox(0,0)[lb]{\smash{{\SetFigFont{12}{14.4}{\familydefault}{\mddefault}{\updefault}{\color[rgb]{0,0,0}$\xout{7}$}%
}}}}
\put(9751,-4936){\makebox(0,0)[lb]{\smash{{\SetFigFont{12}{14.4}{\familydefault}{\mddefault}{\updefault}$\xout{11,12}$}}}}
\put(5176,-6511){\makebox(0,0)[lb]{\smash{{\SetFigFont{12}{14.4}{\familydefault}{\mddefault}{\updefault}$\xout{1}$}}}}
\put(7351,-6511){\makebox(0,0)[lb]{\smash{{\SetFigFont{12}{14.4}{\familydefault}{\mddefault}{\updefault}{\color[rgb]{0,0,0}$\xout{6}$}%
}}}}
\put(10726,-4411){\makebox(0,0)[lb]{\smash{{\SetFigFont{12}{14.4}{\familydefault}{\mddefault}{\updefault}$8,9,\xout{10}$}}}}
\put(7576,-4486){\makebox(0,0)[lb]{\smash{{\SetFigFont{12}{14.4}{\familydefault}{\mddefault}{\updefault}{\color[rgb]{0,0,0}pruned $U(12)$ for $\beta=3,\alpha = -2$}%
}}}}
\put(7576,-5011){\makebox(0,0)[lb]{\smash{{\SetFigFont{12}{14.4}{\familydefault}{\mddefault}{\updefault}$\xout{3,4},5$}}}}
\put(4801,-4186){\makebox(0,0)[lb]{\smash{{\SetFigFont{12}{14.4}{\familydefault}{\mddefault}{\updefault}{\color[rgb]{0,0,0}initial step}%
}}}}
\put(6226,-5536){\makebox(0,0)[lb]{\smash{{\SetFigFont{12}{14.4}{\familydefault}{\mddefault}{\updefault}{\color[rgb]{0,0,0}deletion step}%
}}}}
\put(10426,-5311){\makebox(0,0)[lb]{\smash{{\SetFigFont{12}{14.4}{\familydefault}{\mddefault}{\updefault}{\color[rgb]{0,0,0}correction step}%
}}}}
\put(10426,-5536){\makebox(0,0)[lb]{\smash{{\SetFigFont{12}{14.4}{\familydefault}{\mddefault}{\updefault}{\color[rgb]{0,0,0}delete last}%
}}}}
\put(10426,-5791){\makebox(0,0)[lb]{\smash{{\SetFigFont{12}{14.4}{\familydefault}{\mddefault}{\updefault}{\color[rgb]{0,0,0}$\beta$ labels}%
}}}}
\put(5476,-5011){\makebox(0,0)[lb]{\smash{{\SetFigFont{12}{14.4}{\familydefault}{\mddefault}{\updefault}{\color[rgb]{0,0,0}$\mathbf{\xout{13,14},15}$}%
}}}}
\end{picture}%

%% file: nover2p.pstex_t
\begin{picture}(0,0)%
\includegraphics{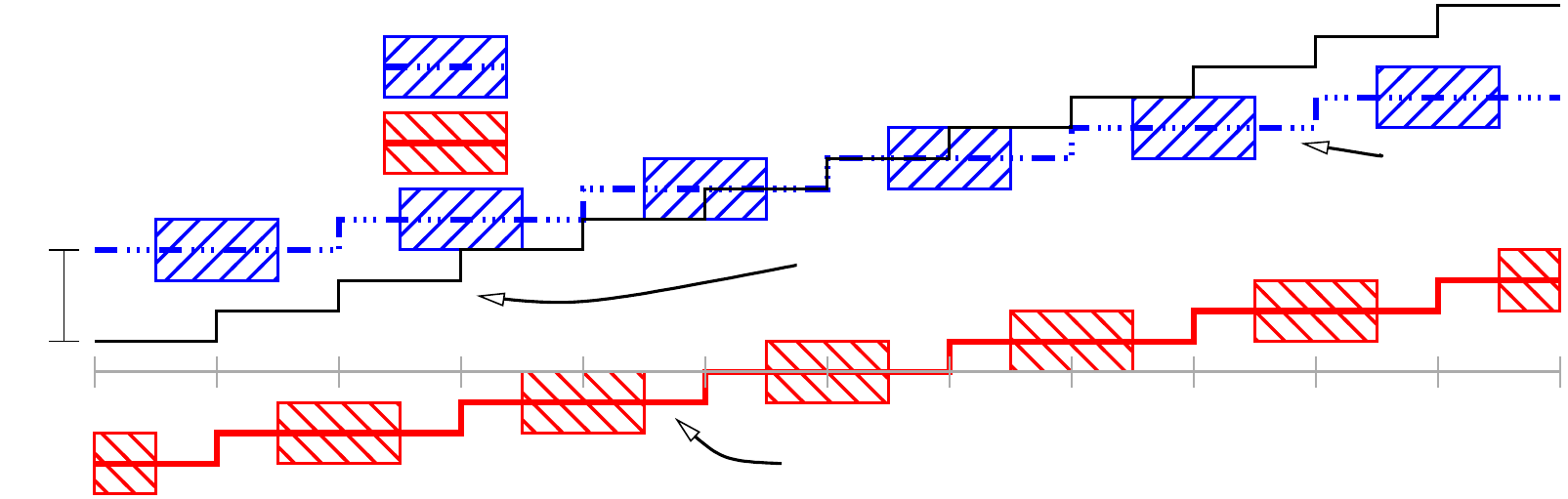}%
\end{picture}%
\setlength{\unitlength}{3947sp}%
\begingroup\makeatletter\ifx\SetFigFont\undefined%
\gdef\SetFigFont#1#2#3#4#5{%
  \reset@font\fontsize{#1}{#2pt}%
  \fontfamily{#3}\fontseries{#4}\fontshape{#5}%
  \selectfont}%
\fi\endgroup%
\begin{picture}(7698,2444)(2986,-4433)
\put(3001,-2386){\makebox(0,0)[lb]{\smash{{\SetFigFont{12}{14.4}{\familydefault}{\mddefault}{\updefault}{\color[rgb]{0,0,1}Lemmas \ref{lem:h1conditions} and \ref{lem:h1const}}%
}}}}
\put(3001,-2761){\makebox(0,0)[lb]{\smash{{\SetFigFont{12}{14.4}{\familydefault}{\mddefault}{\updefault}{\color[rgb]{1,0,0}Lemmas \ref{lem:h2conditions} and \ref{lem:h2const}}%
}}}}
\put(3976,-4036){\makebox(0,0)[lb]{\smash{{\SetFigFont{12}{14.4}{\familydefault}{\mddefault}{\updefault}{\color[rgb]{0,0,0}$2p$}%
}}}}
\put(8251,-4111){\makebox(0,0)[lb]{\smash{{\SetFigFont{12}{14.4}{\familydefault}{\mddefault}{\updefault}{\color[rgb]{0,0,0}$16p$}%
}}}}
\put(8851,-4111){\makebox(0,0)[lb]{\smash{{\SetFigFont{12}{14.4}{\familydefault}{\mddefault}{\updefault}{\color[rgb]{0,0,0}$18p$}%
}}}}
\put(9451,-4111){\makebox(0,0)[lb]{\smash{{\SetFigFont{12}{14.4}{\familydefault}{\mddefault}{\updefault}{\color[rgb]{0,0,0}$20p$}%
}}}}
\put(10051,-4111){\makebox(0,0)[lb]{\smash{{\SetFigFont{12}{14.4}{\familydefault}{\mddefault}{\updefault}{\color[rgb]{0,0,0}$22p$}%
}}}}
\put(10651,-4111){\makebox(0,0)[lb]{\smash{{\SetFigFont{12}{14.4}{\familydefault}{\mddefault}{\updefault}{\color[rgb]{0,0,0}$24p$}%
}}}}
\put(6901,-3211){\makebox(0,0)[lb]{\smash{{\SetFigFont{12}{14.4}{\familydefault}{\mddefault}{\updefault}{\color[rgb]{0,0,0}Assume $h(n) = \cln{n}{2p}$}%
}}}}
\put(4651,-3736){\makebox(0,0)[lb]{\smash{{\SetFigFont{12}{14.4}{\familydefault}{\mddefault}{\updefault}{\color[rgb]{0,0,0}$4p$}%
}}}}
\put(3376,-4036){\makebox(0,0)[lb]{\smash{{\SetFigFont{12}{14.4}{\familydefault}{\mddefault}{\updefault}{\color[rgb]{0,0,0}$0$}%
}}}}
\put(5251,-3736){\makebox(0,0)[lb]{\smash{{\SetFigFont{12}{14.4}{\familydefault}{\mddefault}{\updefault}{\color[rgb]{0,0,0}$6p$}%
}}}}
\put(5851,-3736){\makebox(0,0)[lb]{\smash{{\SetFigFont{12}{14.4}{\familydefault}{\mddefault}{\updefault}{\color[rgb]{0,0,0}$8p$}%
}}}}
\put(7651,-4111){\makebox(0,0)[lb]{\smash{{\SetFigFont{12}{14.4}{\familydefault}{\mddefault}{\updefault}{\color[rgb]{0,0,0}$14p$}%
}}}}
\put(9826,-2836){\makebox(0,0)[lb]{\smash{{\SetFigFont{12}{14.4}{\familydefault}{\mddefault}{\updefault}{\color[rgb]{0,0,1}$h_1(n)$}%
}}}}
\put(3001,-3511){\makebox(0,0)[lb]{\smash{{\SetFigFont{12}{14.4}{\familydefault}{\mddefault}{\updefault}{\color[rgb]{0,0,0}$d$}%
}}}}
\put(6826,-4261){\makebox(0,0)[lb]{\smash{{\SetFigFont{12}{14.4}{\familydefault}{\mddefault}{\updefault}{\color[rgb]{1,0,0}$h_2(n)$}%
}}}}
\end{picture}%

%% file: ConollyLike.bbl
\begin{thebibliography}{99}

\bibitem{AllenbySmith}
R.B.J.T. Allenby and R.C. Smith,
Some sequences resembling Hofstadter's
\emph{J. Korean Math. Soc.} 40 (2003) 921--932.

\bibitem{ShallitAllouche}
Jean-Paul Allouche and Jeffrey Shallit,
\emph{Automatic Sequences: Theory, Applications, Generalizations},
Cambridge University Press, 2003.

\bibitem{BKT}
B. Balamohan, A. Kuznetsov, and S. Tanny,
On the behavior of a variant of Hofstadter's Q-sequence,
  \emph{J. of Integer Sequences} 10 (2007), Article 07.7.1,
  electronic, 29 pages.

\bibitem{BLT} B. Balamohan, Z. Li, and S. Tanny, A
    combinatorial interpretation for certain relatives of the
    Conolly sequence, \emph{J. of Integer Sequences} 11 (2008),
    Article 08.2.1, electronic, 13 pages.

\bibitem{CCT}
Joseph Callaghan, John J. Chew III, Stephen M. Tanny,
On the behavior of a family of meta-Fibonacci sequences,
\emph{SIAM J. Discrete Math.} 18(4) (2005), 794--824.

\bibitem{ConVa} B. W. Conolly, Meta-Fibonacci sequences, in S.
    Vajda, ed., \emph{Fibonacci \& Lucas Numbers, and the
    Golden Section}, Wiley, New York, 1986, pp. 127-137.

\bibitem{Golomb1990}
Solomon W. Golomb,
  Discrete chaos: sequences satisfying strange recursions, preprint, undated.

\bibitem{GKP} R. Graham, D. E. Knuth, and O. Patashnik, \emph{Concrete Mathematics}, Second Edition, Addison-Wesley (1994).

  
\bibitem{HiTan} J. Higham and S. Tanny, More well-behaved
    meta-Fibonacci sequences, \emph{Congr. Numer.} 98 (1993),
    3-17.

\bibitem{Hofstadter}
Douglas R. Hofstadter,
  \emph{Godel, Escher, and Bach: An Eternal Golden Braid}, Basic Books, New York, 1979.


\bibitem{Rpaper} Abraham Isgur, David Reiss, and Stephen Tanny, Trees and meta-Fibonacci sequences, \emph{Electron. J. of Combin.} 16 (2009), R129.


\bibitem{knuthvol1} D. E. Knuth,
  \newblock \emph{The Art of Computer
  Programming}. Volume 1, Fundamental Algorithms, Third Edition,
  \newblock Addison-Wesley (1997).

  
\bibitem{JacksonRuskey}
  B. Jackson and F. Ruskey, Meta-Fibonacci sequences, binary trees and extremal compact codes,
    \emph{Electron. J. of Combin.} 13 (2006), R26.



  \bibitem{FRusCDeg} F. Ruskey and C. Deugau, The combinatorics of
    certain $k$-ary meta-Fibonacci sequences, \emph{Journal of Integer
      Sequences} 12, Article 09.4.3, 36 pages, 2009.

\bibitem{sna} N. J. A. Sloane, \emph{Online Encyclopedia of
    Integer Sequences}, \hfil\break
    \href{http://www.research.att.com/~njas/sequences}{\tt
    http://www.research.att.com/$\sim$njas/sequences}.

\bibitem{Stoll}
T. Stoll,
On Hofstadter's married functions,
  \emph{Fibonacci Quarterly} 46/47 (2008/2009), 62-67.

\bibitem{Tanny92}
S.M. Tanny,
A well-behaved cousin of the Hofstadter sequence,
\emph{Discrete Mathematics} 105 (1992) 227--239.





\end{thebibliography}
